\documentclass[11pt]{article}

\usepackage{multirow}
\usepackage{amsmath, amsfonts, amsthm, verbatim, amssymb}

\newcommand{\dd}{\mathrm{d}}
\newcommand{\R}{\mathbb{R}}
\newcommand{\N}{\mathbb{N}}
\newcommand{\E}{\mathbb{E}}
\newcommand{\p}{\mathbb{P}}

\usepackage{tikz}
\usetikzlibrary{arrows.meta,arrows}

\newtheoremstyle{neu-theorem}
{11pt}      
{11pt}      
{\itshape}                  
{}          
{\bfseries} 
{}          
{1em}  
{\textbf{\thmname{#1}\thmnumber{ #2}\thmnote{ (#3)}.}}          

\theoremstyle{neu-theorem}

\newtheorem{theorem}{Theorem}[section]
\newtheorem{corollary}[theorem]{Corollary}
\newtheorem{lemma}[theorem]{Lemma}
\newtheorem{proposition}[theorem]{Proposition}
\newtheorem{theoremA}{Theorem}

\newtheoremstyle{neu}
{11pt}      
{11pt}      
{}                  
{}          
{\bfseries} 
{}          
{1em}  
{\textbf{\thmname{#1}\thmnumber{ #2}\thmnote{ (#3)}.}}          

\theoremstyle{neu}
\newtheorem{definition}[theorem]{Definition}
\newtheorem{example}[theorem]{Example}
\newtheorem{remark}[theorem]{Remark}
\newtheorem{assumption}{Assumption}

\usepackage{latexsym}
\usepackage{amssymb}
\usepackage{amsmath}
\usepackage{amsthm}
\usepackage{mathrsfs}
\usepackage{url}
\usepackage{longtable}
\usepackage{color}
\usepackage{dsfont}
\usepackage[T1]{fontenc}
\usepackage[utf8]{inputenc}
\usepackage{lmodern}
\usepackage[linktocpage]{hyperref}

\usepackage{graphicx}

\usepackage{amsopn}

\newcommand{\bfi}{\begin{fig}}
	\newcommand{\efi}{\end{fig}}
\newcommand{\btab}{\begin{tab}}
	\newcommand{\etab}{\end{tab}}
\newcommand{\barr}{\begin{array}}
	\newcommand{\earr}{\end{array}}
\newcommand{\beq}{\begin{equation}}
\newcommand{\eeq}{\end{equation}}
\newcommand{\bdis}{\begin{displaymath}}
\newcommand{\edis}{\end{displaymath}\noindent}

\newcommand{\bbn}{\mathbb{N}}

\newcommand{\bbr}{\mathbb{R}}

\newcommand{\bbe}{\mathbb{E}}

\newcommand{\bbp}{\mathbb{P}}

\newcommand{\bone}{\mathds 1}

\newcommand{\cali}{{\cal I}}

\newcommand{\al}{{\alpha}}
\newcommand{\la}{{\lambda}}
\newcommand{\La}{{\Lambda}}
\newcommand{\eps}{{\varepsilon}}

\newcommand{\si}{{\sigma}}
\newcommand{\om}{{\omega}}

\newcommand{\ov}{\overline}

\newcommand{\Comp}{\mathrm{c}}

\definecolor{orange}{rgb}{1,0.5,0}

\newcommand{\Leb}{\mathrm{Leb}}

\newcommand{\lb}{\left[}
\newcommand{\rb}{\right]}
\newcommand{\lp}{\left(}
\newcommand{\rp}{\right)}
\newcommand{\lv}{\left|}
\newcommand{\rv}{\right|}
\newcommand{\as}{\quad\text{a.s.}}

\newcommand{\bthm}{\begin{theorem}}
	\newcommand{\ethm}{\end{theorem}}

\newcommand{\bthmA}{\begin{theoremA}}
	\newcommand{\ethmA}{\end{theoremA}}

\newcommand{\bcor}{\begin{corollary}}
	\newcommand{\ecor}{\end{corollary}}

\newcommand{\blem}{\begin{lemma}}
	\newcommand{\elem}{\end{lemma}}

\newcommand{\bprop}{\begin{proposition}}
	\newcommand{\eprop}{\end{proposition}}

\newcommand{\bdf}{\begin{definition}}
	\newcommand{\edf}{\end{definition}}

\newcommand{\bex}{\begin{example}}
	\newcommand{\eex}{\end{example}}

\newcommand{\brem}{\begin{remark}}
	\newcommand{\erem}{\end{remark}}

\newcommand{\bass}{\begin{assumption}}
	\newcommand{\eass}{\end{assumption}}

\newcommand{\bpr}{\begin{proof}}
	\newcommand{\epr}{\end{proof}}

\newcommand{\benu}{\begin{enumerate}}
	\newcommand{\eenu}{\end{enumerate}}

\newcommand{\bit}{\begin{itemize}}
	\newcommand{\eit}{\end{itemize}}

\newcommand{\itt}{\textit}

\numberwithin{equation}{section}

\title{The almost-sure asymptotic behavior of the solution to
the stochastic heat equation with Lévy noise}

\author{
Carsten Chong\thanks{Institut de math\'ematiques, \'Ecole Polytechnique F\'ed\'erale de 
Lausanne, Station 8, CH-1015 Lausanne, e-mail: carsten.chong@epfl.ch} \and P\'eter 
Kevei\thanks{{Bolyai Institute, University of Szeged,
Aradi v\'ertan\'uk tere 1, 6720 Szeged, Hungary, 
e-mail: kevei@math.u-szeged.hu}}
}

\date{}

\begin{document}
\maketitle
\begin{abstract} 
We examine the almost-sure asymptotics of the solution to the stochastic heat equation 
driven by a Lévy space-time white noise. When a spatial point is fixed and time tends to infinity, we 
show that the solution develops unusually high peaks over short time intervals, even in the case of additive noise, which leads to 
a breakdown of an intuitively expected strong law of large numbers. More precisely, if 
we normalize the solution by an increasing nonnegative function, we either obtain 
convergence to $0$, or the limit superior and/or inferior will be infinite.
A detailed analysis of the jumps further reveals that the strong law of large 
numbers can be recovered on discrete sequences of time points increasing to infinity. This 
leads to a necessary and sufficient condition that depends on the Lévy measure of the 
noise and the growth and concentration properties of the sequence at the same time. 
Finally, we show that our results generalize to the stochastic heat equation with a 
multiplicative nonlinearity that is bounded away from zero and infinity.
\end{abstract}

\vspace{\baselineskip}
\noindent
\begin{tabbing}
\end{tabbing}
\vspace{\baselineskip}

\vspace{\baselineskip}

\section{Introduction}

Consider the stochastic heat equation on $\R^d$ driven by a Lévy space-time  
white noise  $\dot \La$, with zero initial condition:
\beq\label{SHE}
\begin{split}
\partial_t Y(t,x) &= \Delta Y(t,x) + \si(Y(t,x))\dot\La (t,x),\qquad 
(t,x)\in(0,\infty)\times\bbr^d,\\
Y(0,\cdot)&= 0,
\end{split}
\eeq
where $\si\colon\R\to(0,\infty)$ is a Lipschitz continuous function that is bounded away from $0$ and infinity. The purpose of this paper is to report on some unexpected 
asymptotics of the solution $Y(t,x)$, for some fixed spatial point $x\in\R^d$, as time 
tends to infinity.

In order to describe the atypical behavior we encounter, let us consider in this introductory part the simplest possible situation where $\si\equiv1$ and $\dot \La$ is a standard Poisson noise, that is, $\dot \La = \sum_{i=1}^\infty \delta_{(\tau_i,\eta_i)}$ is a sum of Dirac delta functions at random space-time points $(\tau_i,\eta_i)$ that are determined by a standard Poisson point process on $[0,\infty)\times\R^d$. In this case, $Y(t,\cdot)$ can be interpreted as the density at time $t$ of a random measure describing particles that are placed according to the point process and perform independent $d$-dimensional Brownian motions.

As the mild solution $Y$ to \eqref{SHE} in this simplified case takes the form 
\beq\label{eq:mild}
Y(t,x)  = \int_0^t \int_{\R^d} g(t-s, x-y)\, \Lambda ( \dd s, \dd y) = \sum_{i=1}^\infty g(t-\tau_i,x-\eta_i)\bone_{\tau_i<t},
\eeq
where 
\begin{equation} \label{eq:def-heat}
g(t,x) = \frac{1}{(4 \pi t)^{d/2}} e^{-\frac{|x|^2}{4 t}}\bone_{t>0}, \qquad
(t,x) \in [0,\infty)\times \R^d,
\end{equation}
is the heat kernel in $d$ dimensions ($|\cdot|$ denotes the Euclidean norm in $\R^d$), we immediately see that 
\beq\label{eq:exp} 
\bbe[Y(t,x)] = \int_0^t\int_{\R^d} g(t-s,x-y)\,\dd y\,\dd s
= \int_0^t 1 \,\dd s = t,\qquad 
(t,x)\in[0,\infty)\times\R^d.
\eeq
Hence, one expects to have a \emph{strong law of large numbers} (SLLN) as $t\to\infty$ in the sense that for fixed $x\in\R^d$, we have
\[ \lim_{t\to\infty} \frac{Y(t,x)}{t}  = 1 \quad\text{a.s.} \]

The starting point of this paper is the observation that the last statement turns out to 
be \emph{false}. Let us consider without loss of generality the point $x=(0,\dots,0)$ 
and write 
\beq\label{eq:Y0} Y_0(t)=Y(t,0),\eeq 
which is a process with almost surely smooth sample paths by \cite[Théorème~2.2.2]{SLB98}.
\bthmA\label{thm:A} 
Let $f\colon (0,\infty) \to (0, \infty)$ be a nondecreasing function, $\si\equiv1$, $Y_0$ be given by \eqref{eq:Y0}, and $\dot \La$ be a 
standard Poisson noise. Then, with probability one, we have
\[
\limsup_{t \to \infty} \frac{Y_0(t)}{f(t)} = \infty \qquad\text{or}\qquad 
\limsup_{t \to \infty} \frac{Y_0(t)}{f(t)} =0,
\]
according to whether 
\[
\int_1^\infty \frac{1}{f(t)}\, \dd t = \infty\qquad \text{or}\qquad \int_1^\infty 
\frac{1}{f(t)}\, \dd t < \infty.
\]
Furthermore, we almost surely have
\[
\liminf_{t \to \infty} \frac{Y_0(t)}{t} = 1.
\]
\ethmA
In other words, while the limit inferior follows the expected SLLN, 
the integral test for the limit superior shows that there is no natural nonrandom normalization that would ensure a nontrivial limit. 
For example, we have 
\[ \limsup_{t\to\infty} \frac{Y_0(t)}{t} = \limsup_{t\to\infty} \frac{Y_0(t)}{t\log t} = \infty\qquad \text{but}\qquad \limsup_{t\to\infty} \frac{Y_0(t)}{t(\log t)^{1.1}} =0 \]
almost surely
 This kind of 
 phenomenon is common for stochastic processes with \itt{infinite expectation}; 
 see, for instance, \cite[Theorem 2]{ChowRobbins} for the case of i.i.d.\ sums and
 \cite[Theorem III.13]{Bertoin} for the case of subordinators (i.e., nonnegative Lévy processes). But it is unusual in our case because $Y_0$ does have a finite 
 expectation by \eqref{eq:exp} (in fact, even a finite variance if $d=1$). 
 
With Gaussian noise, we do not have such irregular behavior but a proper limit theorem:
\bthmA\label{thm:B} Suppose that $\si\equiv1$ and $\dot \La$ is a Gaussian space-time white noise in one spatial dimension. Then the following \emph{law of the iterated logarithm} holds for $Y_0$ from \eqref{eq:Y0}:
\[ \limsup_{t\to\infty} \frac{Y_0(t)}{(2t/\pi)^{1/4}\sqrt{\log \log t}} = -\liminf_{t\to\infty} \frac{Y_0(t)}{(2t/\pi )^{1/4}\sqrt{\log \log t}} = 1\as \]
In particular, the SLLN holds:
\[ \lim_{t\to\infty} \frac{Y_0(t)}{t} = 0\as \]
\ethmA
This theorem follows easily from the general theory on the growth of Gaussian processes \cite{Watanabe70}. Although it is known that $Y_0(t)$ in the Gaussian case locally looks  like a fractional Brownian motion with Hurst parameter $\frac14$ (see \cite[Theorem~3.3]{Khos}), we could not find the corresponding global statement specifically for the stochastic heat equation. Hence, we give a short proof of Theorem~\ref{thm:B} at the end of Section~\ref{sect:mr}.

Back to the Lévy case, the exact sample path behavior of $Y_0$ is even more complex than described by Theorem~\ref{thm:A}. Our proofs will reveal that the failure of the SLLN for the limit superior is due to the jumps that occur in a short space-time distance to $(t,x)=(t,0)$.
However, these problematic jumps that cause the deviation from the SLLN only have a very short impact. In fact, if we only observe $Y_0$ on discrete time points, say, at $t=t_n=n^p$, for $n\in\N$ and some $p>0$, then we have the following result:
\bthmA\label{thm:C} For $Y_0$ from \eqref{eq:Y0}, if $t_n=n^p$ for some $p>d/(d+2)$, we have
\[ \lim_{n\to\infty} \frac{Y_0(t_n)}{t_n}  = 1\quad\text{a.s.},\]
while for $0<p\leq d/(d+2)$, we have
\[ \limsup_{n\to\infty} \frac{Y_0(t_n)}{t_n}  = \infty\qquad\text{and}\qquad \liminf_{n\to\infty} \frac{Y_0(t_n)}{t_n}  = 1\quad\text{a.s.}\] 
\ethmA

So if we sample the solution on a fast sequence (``large $p$''), those problematic jumps are not visible, and the SLLN \emph{does} hold true. If the sequence is too slow (``small $p$''), they are visible (as in the continuous-time case), and the SLLN fails. Let us make the following observations:
\benu
\item The SLLN holds on the sequence $t_n = n$ in any dimension.
\item For any $0<p<1$, the SLLN will fail on the sequence $t_n=n^p$ in sufficiently high dimension.
\item For any dimension $d\geq1$, if we take $p\in(d/(d+2),1)$, we obtain with $t_n=n^p$ a sequence whose increments $\Delta t_n = t_n-t_{n-1} = n^p-(n-1)^p = O(n^{p-1})$ converge to $0$ as $n\to\infty$, but on which the SLLN still holds. Together with Theorem~\ref{thm:A}, this means that between infinitely many consecutive points of the sequence $t_n$, which get closer and closer and where $Y_0$ is of order $t_n$, there are time points where $Y_0$ is significantly larger.
\eenu

In Figure \ref{fig} we see a simulated path of $t\mapsto Y_0(t)$ for $t \in [0,200]$ and its restriction to the sequences $n$, $n^{0.5}$, and $n^{0.3}$, respectively. While unusually large peaks are clearly visible in the plots of $t\mapsto Y_0(t)$ and $n^{0.3}\mapsto Y_0(n^{0.3})$, only small deviations from the linear growth of $Y_0$ are observed on the sequences $n$ and $n^{0.5}$. This is in agreement with the theoretical considerations above because in dimension $1$, the SLLN holds on the sequence $n^p$ for all $p>\frac13$, while it fails for all $p\leq \frac13$ and in continuous time. 

\begin{figure} \label{fig}
	\includegraphics[width=\linewidth]{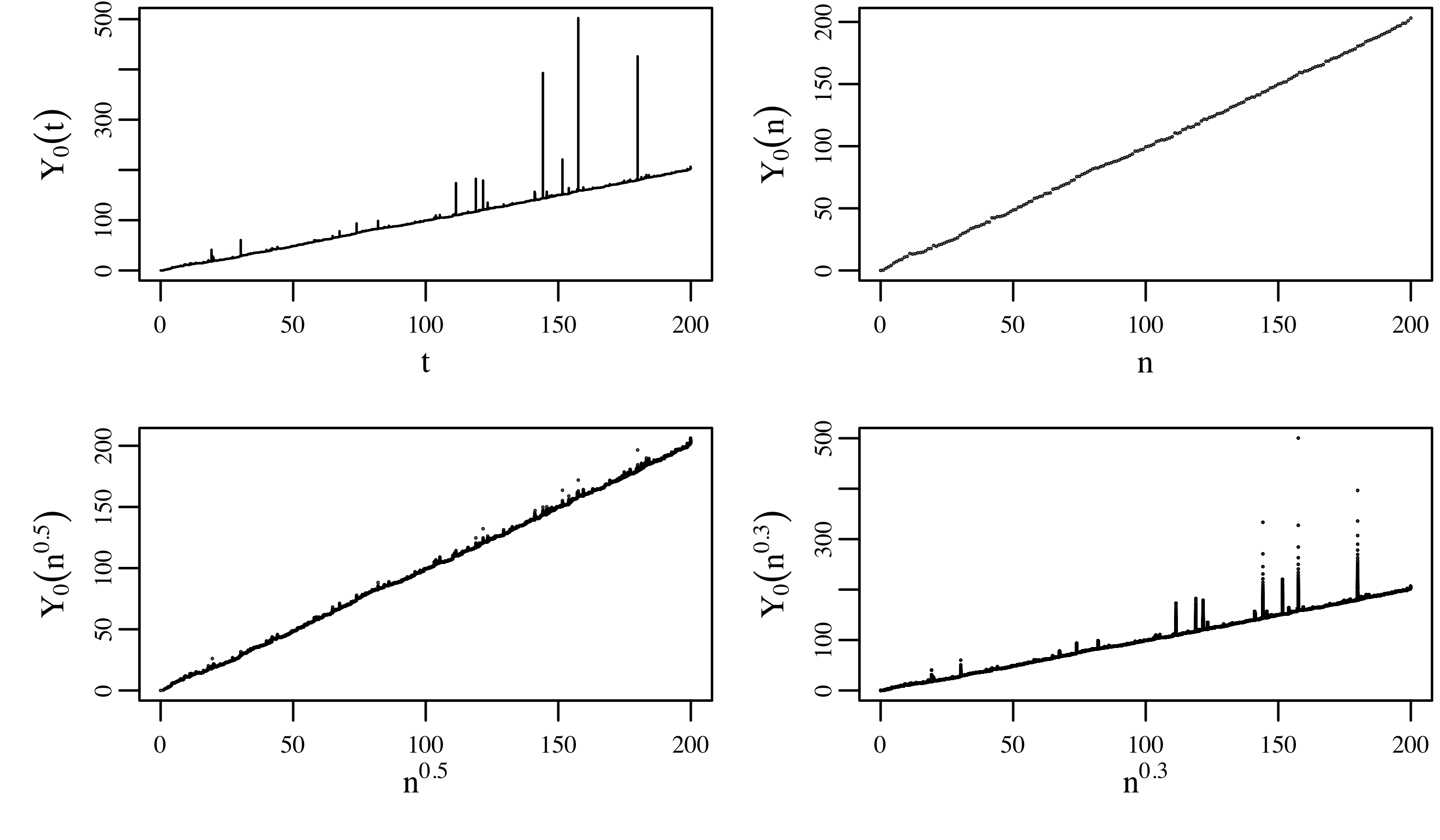}
	%
	
	\vspace{-\baselineskip}
	\caption{A simulated sample path of $t\mapsto Y_0(t)$, $n\mapsto Y_0(n)$, $n^{0.5}\mapsto Y_0(n^{0.5})$, and $n^{0.3}\mapsto Y_0(n^{0.3})$ under the same realization of a standard Poisson noise in dimension $d=1$. Using Proposition~\ref{prop:small-box}, we have approximated the contribution of jumps with a distance of more than $5$ from $x=0$ by the mean. For the plot of $t\mapsto Y_0(t)$, we have used an equidistant grid of step size $0.01$, and have further included for each jump, the time point of the induced local maximum according to Lemma~\ref{lem:hk} in the simulation grid.}
\end{figure}

A similar dichotomy is also found in Figure~\ref{fig2}, which suggests that the averages $Y_0(n)/n$ and $Y_0(n^{0.5})/n^{0.5}$ stabilize at the mean $1$ for large values of $n$, whereas in continuous time or on the sequence $n^{0.3}$, significant deviations from the mean are repeatedly observed at isolated time points.

\begin{figure} 
	\includegraphics[width=\textwidth]{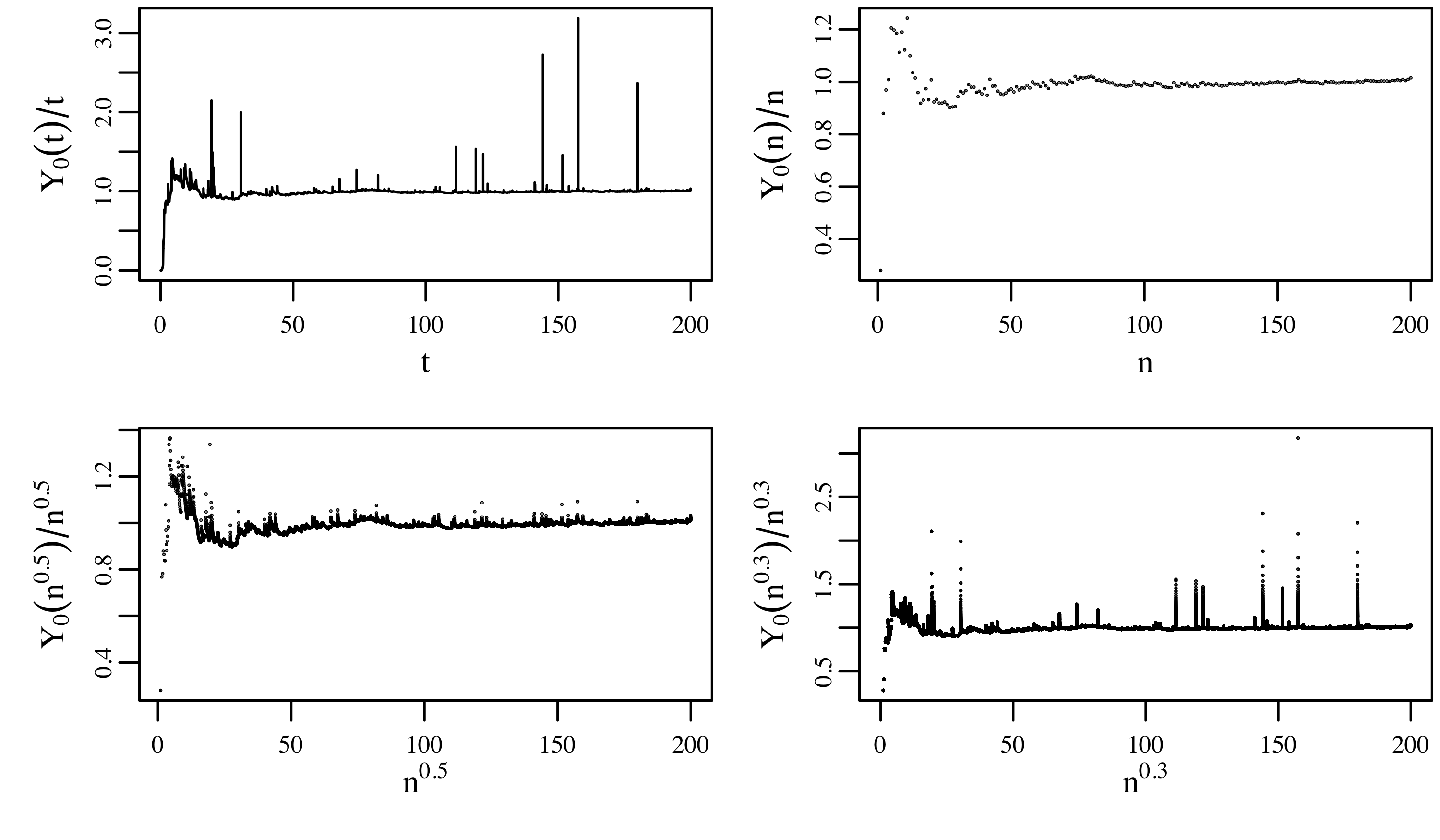}
%

		\vspace{-\baselineskip}
		\caption{The averages $t\mapsto Y_0(t)/t$, $n\mapsto Y_0(n)/n$, $n^{0.5}\mapsto Y_0(n^{0.5})/n^{0.5}$ and $n^{0.3}\mapsto Y_0(n^{0.3})/n^{0.3}$ for the sample path from Figure~\ref{fig}.}
 \label{fig2}
\end{figure}


\medskip

Let us interpret these results in a larger context. In the analysis of random fields, many 
different authors have studied  the phenomenon of \emph{intermittency}. Originating from 
the physics literature on turbulence (see \cite[Chapter~8]{Frisch95}), it refers to the 
chaotic behavior of a random field that develops unusually high peaks over small areas. 

Concerning the stochastic heat equation, it is well known from \cite{Bertini95, Foondun09, 
Khos} that the solution to \eqref{SHE} driven by a Gaussian space-time white noise in dimension $1$ is not intermittent if $\si$ is a bounded function, while it is intermittent if $\si$ has linear growth. Here, intermittency, or 
more precisely, \emph{weak intermittency} is mathematically defined as the exponential 
growth of the moments of the solution. However, the translation of this purely moment-based notion 
of intermittency to a pathwise description of the exponentially large peaks of the solution, sometimes referred to as 
\emph{physical intermittency}, has not been fully resolved yet; see 
\cite[Section~2.4]{Bertini95} or \cite[Chapter~7.1]{Khos} for some heuristic arguments. Despite recent results of 
\cite{Khoshnevisan18} on the multifractal nature of the space-time peaks of the solution, 
the exact almost-sure asymptotics of the solution as time tends to infinity, for fixed 
spatial location, are still unknown. To our best knowledge, only a weak law of large numbers has been proved rigorously for certain initial conditions when $\si$ is a linear function; see \cite{Bertini99} and, in particular, \cite{Amir11, Corwin16}, where much deeper fluctuation results  were obtained. Let us also mention that for fixed time,
 the almost-sure behavior of the solution in 
space  has been resolved in \cite{Conus13, Khoshnevisan17}.

However, in the case of additive Gaussian noise,  Theorem~\ref{thm:B} does reveal the pathwise  asymptotics of the solution: it obeys the law of the iterated logarithm and is therefore not physically intermittent. But what about the case of \emph{Lévy} noise? As Theorem~\ref{thm:A} and the last statement in (3) above show, the solution develops high peaks over very short periods of time. So this leads us to the question:
\[ \text{\emph{Is the solution $Y$ to \eqref{SHE} with Lévy noise physically intermittent?}}\]
Certainly not in the sense of exponential growth of the solution because $\si$ is a bounded function (so multiplicative effects cannot build up). But it seems appropriate to say that $Y$ exhibits \emph{additive physical intermittency}. We use the attribute ``additive'' to describe the fact that the tall peaks of $Y_0(t)$ do not arise through a multiplicative cascade of jumps, or the accumulation of past peaks, but rather through the effect of single isolated jumps. 

That additive physical intermittency only occurs with jump noise, but not with Gaussian 
noise, is in line with \cite{Chong18}, where we have shown that for the heat equation with 
multiplicative Lévy noise, weak intermittency occurs on a much larger scale than under 
Gaussian noise.

Let us mention that a weak (i.e., moment-based) version of additive intermittency has 
been introduced in a series of papers \cite{Grahovac18,Grahovac16,Grahovac18a} on 
superpositions of Ornstein--Uhlenbeck processes. The term ``additive intermittency'' 
itself was coined by  Murad S. Taqqu in private communication with the first author discussing the references above.

The remaining paper is organized as follows. In Section~\ref{sect:results}, we will 
describe our main results concerning the asymptotic behavior of $Y$ in continuous time 
 as well as on discrete subsequences. The case of additive Lévy noise will be investigated in Theorems~\ref{thm:cont}, \ref{thm:WLLN}, and \ref{thm:Yf-ds}, respectively. 
Special cases will be discussed in Corollaries~\ref{cor:1} and \ref{cor:2} and 
Examples~\ref{ex:1} and \ref{ex:2} in order to illustrate the subtle necessary and 
sufficient conditions found in these theorems. In 
Theorems~\ref{thm:cs-g}, \ref{thm:WLLN-g}, and \ref{thm:ds-g}, we then extend the results to the stochastic heat 
equation with multiplicative noise when the nonlinear function $\si$ is bounded away from zero 
and infinity. The proofs will be given in Section~\ref{sect:proof}, where we analyze the ``bad'' jumps (that could 
destroy the SLLN) and the 
``nice'' jumps (that behave according to the SLLN) separately in Sections~\ref{sect:rc} and \ref{sect:of}, before proving 
the main results in Section~\ref{sect:mr}.

Throughout this paper, we use $C$ to denote a strictly positive finite constant whose 
exact value is not important and may change from line to line.

\section{Results}\label{sect:results}

As the Gaussian case is studied separately in Theorem~\ref{thm:B},
we assume from now on that $\dot \La$ is a Lévy space-time white noise without Gaussian 
part. More specifically, we suppose that the random measure associated to $\dot \La$ is 
given by
\begin{equation} \label{eq:def-Lambda}
\begin{split}
\Lambda(\dd t, \dd x) & = m\, \dd t\, \dd x + 
\int_{\R} z\, (\mu - \nu)(\dd t, \dd x, \dd z) = m_0 \,\dd t\, \dd x
+ \int_{\R} z  \,\mu(\dd t, \dd x, \dd z),
\end{split}
\end{equation}
where 
$m\in\R$ is the mean of the noise $\dot \La$,
$\mu$ is a Poisson random measure on 
$(0, \infty) \times \R^d \times \R$ whose intensity measure $\nu$ takes the form
$\nu ( \dd t, \dd x, \dd z) = \dd t\, \dd x\, \lambda(\dd z)$,
with a L\'evy measure $\lambda$ satisfying
\begin{equation} \label{eq:int-lambda}
\int_\R |z| \,\lambda (\dd z) < \infty,
\end{equation}
and $m_0=m-\int_{\R} z \,\la(\dd z)$ is the \emph{drift} of $\dot \La$. In particular, for bounded Borel sets $A,A_i\subseteq[0,\infty)\times\R^d$ such that $(A_i)_{i\in\N}$ are pairwise disjoint, the random variables $(\La(A_i))_{i\in\N}$ are independent, and we have the Lévy--Khintchine formula
\[ \E[e^{iu\La(A)}] = \exp\lp\lp im_0 u + \int_\R (e^{iuz}-1)\,\la(\dd z) \rp\Leb(A)\rp,\qquad u\in\R, \]
where $\Leb$ denotes the Lebesgue measure on $[0,\infty)\times\R^d$.
In what follows, we 
always assume that $\lambda$ is not identically zero.

\brem 
Condition \eqref{eq:int-lambda} makes sure that the jumps of the Lévy noise are locally 
summable. In fact, if $\int_{[-1,1]} |z|^p\,\la(\dd z) = \infty$ for some $p>1$, then the 
sample path $t\mapsto Y(t,x)$, for fixed $x\in\R^d$, is typically unbounded on any 
nonempty open subset of $[0,\infty)$; see \cite[Theorem~3.7]{Chong18b}. In this case, if 
the noise has jumps of both signs, we trivially have
\[
\limsup_{t \to \infty} \frac{Y_0(t)}{f(t)} = -\liminf_{t \to \infty} 
\frac{Y_0(t)}{f(t)} =\infty
\]
for any nonnegative nondecreasing function $f$, due to the local irregularity of the 
solution. As the focus of this paper is on the global irregularity of the solution, we 
will assume \eqref{eq:int-lambda} in all what follows. In particular, by 
\cite[Theorem~3.5]{Chong18b}, if there exists $0<p<1$ such that
$\int_{[-1,1]} |z|^p\,\la(\dd z) < \infty$, then $t\mapsto Y(t,x)$, for fixed $x\in\R^d$, 
has a continuous modification.
\erem

\subsection{Additive noise}
We first consider the case of additive Lévy noise.
It is immediate to see that under the assumption \eqref{eq:int-lambda}, the mild solution 
to \eqref{SHE} given by \eqref{eq:mild}
is well defined. As in the introduction, we shall fix a spatial point, say 
$x=(0,\dots,0)$, and investigate the behavior of $Y_0(t) = Y(t,0)$ as $t\to\infty$. 
The 
following result extends Theorem~\ref{thm:A} to general Lévy noise, assuming a 
slightly stronger condition than \eqref{eq:int-lambda}:
\beq\label{eq:cond-lambda}
\exists\, \eps>0\colon m_\la(1+\eps)  = \int_\R
|z|^{1+\eps}\,\la(\dd z)<\infty\qquad\text{and}\qquad \int_{-1}^1 |z \log z|\,\la(\dd z) <\infty.
\eeq
\begin{theorem} \label{thm:cont}
Let 
$f\colon (0,\infty) \to (0,\infty)$ be a nondecreasing function, $\si\equiv1$, $Y_0$ be given by \eqref{eq:Y0}, and $\la$ satisfy 
\eqref{eq:cond-lambda}.
\benu
\item If $\int_1^\infty 1/f(t) \,\dd t = \infty$ and $\la((0,\infty))>0$ (resp., 
$\la((-\infty,0))>0$), then
\[
\limsup_{t \to \infty} \frac{Y_0(t)}{f(t)} = \infty  \qquad \left(\text{resp.,}\quad 
\liminf_{t \to \infty} \frac{Y_0(t)}{f(t)} = -\infty\right)\quad\as
\]
\item Conversely, if 
$\int_1^\infty 1/f(t) \,\dd t < \infty$,
 then
\[
\lim_{t \to \infty} \frac{Y_0(t)}{f(t)} = 0 \quad \text{a.s.} 
\]
\item If $\la((0,\infty))=0$ (resp., $\la((-\infty,0))=0$), then
\beq\label{eq:liminf}
\limsup_{t \to \infty} \frac{Y_0(t)}{t} = m  \qquad \left(\text{resp.,}\quad \liminf_{t \to 
\infty} \frac{Y_0(t)}{t} = m\right)\quad\as
\eeq
\eenu
\end{theorem}

Hence, the SLLN fails for any non-Gaussian Lévy noise. 
Let us remark, however, that the weak law of large numbers does hold true. In particular, there is no (additively) intermittent behavior of the moments of the solution! In the following result, we only consider moments of order less than $1+2/d$ because all higher moments are infinite by \cite[Theorem~3.1]{Chong18}.
\bthm\label{thm:WLLN}
Suppose that $\si\equiv1$, that $Y_0$ is given by \eqref{eq:Y0}, and that $\la$ satisfies \eqref{eq:cond-lambda} with some $\eps>0$. Then for every $p\in (0,1+\eps]\cap(0,1+2/d)$, we have 
\beq\label{eq:WLLN} \frac{Y_0(t)}{t} \stackrel{L^p}{\longrightarrow} m\qquad \text{as}\qquad t\to\infty. \eeq
\ethm

Next, we continue with our 
discussion on subsequences. 
The following theorem extends Theorem~\ref{thm:C} to general Lévy noises as well as 
general sequences and weight functions. Recall the notation $\Delta t_n = t_n - t_{n-1}$ 
(with $t_0=0$) for the increments of $t_n$.
\bthm\label{thm:Yf-ds}
Let $f\colon(0,\infty)\to(0,\infty)$ be a nondecreasing function, $\si\equiv1$,  $Y_0$ be given by \eqref{eq:Y0}, and $t_n$ be a 
nondecreasing sequence tending to infinity. Assume that $\la$ satisfies 
\eqref{eq:cond-lambda}.
\benu
\item If
\begin{align} \label{eq:t-cond-0}
&\sum_{n=1}^\infty \int_0^\infty \lp \lp\frac{z}{f(t_n)}\rp^{2/d}\wedge \Delta t_n 
\rp\frac{z}{f(t_n)}\,\la(\dd z) = \infty\\
\label{eq:t-cond-1} \Bigg(\text{resp.,}\quad &\sum_{n=1}^\infty \int_{-\infty}^0 \lp \lp\frac{|z|}{f(t_n)}\rp^{2/d}\wedge \Delta t_n \rp\frac{|z|}{f(t_n)}\,\la(\dd z) = \infty \Bigg)
\end{align}
and $\liminf_{n \to \infty} \frac{f(t_n)}{t_n} > 0$, then
\beq\label{eq:aux2}
\limsup_{n \to \infty} \frac{Y_0(t_n)}{f(t_n)} = \infty \qquad \left(\text{resp.,}\quad \liminf_{n \to \infty} \frac{Y_0(t_n)}{f(t_n)} = -\infty \right) \quad \as
\eeq
\item Conversely, suppose that the series in \eqref{eq:t-cond-0} (resp., 
\eqref{eq:t-cond-1}) is finite. 
If $f$ is unbounded and
\beq\label{eq:kappa} \lim_{n \to \infty} 
\frac{t_n}{f(t_n)} = \kappa \in [0,\infty], \eeq
where $\kappa<\infty$ when $m=0$, then, with probability $1$,
\beq\label{eq:aux}
\limsup_{n \to \infty} \frac{Y_0(t_n)}{f(t_n)} = \kappa m\qquad \left(\text{resp.,}\quad \liminf_{n \to \infty} \frac{Y_0(t_n)}{f(t_n)} = \kappa m\right).
\eeq
\eenu
\ethm

\brem On the one  hand, as we shall explain in Remark~\ref{rem:growth}, it is a natural condition to require $\liminf_{n \to \infty} f(t_n)/t_n > 0$ in the first part of Theorem~\ref{thm:Yf-ds}. On the other hand, in order that \eqref{eq:t-cond-0} or \eqref{eq:t-cond-1} hold, the function $f$ must not grow very fast, either. Indeed,
if $\int_1^\infty 1/f(t)\, \dd t < \infty$, then 
$\sum_{n=1}^\infty \Delta t_n / f(t_n) < \infty$ by Riemann-sum approximation, so \eqref{eq:t-cond-0} and \eqref{eq:t-cond-1} cannot be true (in agreement with part (2) of Theorem~\ref{thm:cont}). Typical functions that we have in mind are $f(t)=t$ or $f(t)=t\log^+ t$ (where $\log^+(t)=\log(e+t)$).
\erem

Theorem~\ref{thm:Yf-ds} has a number of surprising consequences. We shall explain them as well as 
 the conditions \eqref{eq:t-cond-0} and \eqref{eq:t-cond-1} through a series of corollaries and examples.

If we bound the minimum in \eqref{eq:t-cond-0} (resp., \eqref{eq:t-cond-1}) by the first 
term, we immediately obtain the following result:
\bcor\label{cor:1} If
\beq \label{eq:f}\begin{split}
&\int_0^\infty |z|^{1+2/d}\,\la(\dd z)<\infty \qquad\lp 
\text{resp.,}\quad\int_{-\infty}^0 |z|^{1+2/d}\,\la(\dd z)<\infty\rp\\
 \text{and}\qquad 
&\sum_{n=1}^\infty \frac{1}{f(t_n)^{1+2/d}} < \infty,\end{split}
\eeq
then the series in \eqref{eq:t-cond-0} (resp., \eqref{eq:t-cond-1}) is finite.
\ecor

Upon taking $f(x)=x$ and $t_n=n^p$, we immediately obtain the first statement of Theorem~\ref{thm:C}.
Observe that \eqref{eq:f} separates the complicated expressions in \eqref{eq:t-cond-0} and \eqref{eq:t-cond-1} into a simple \emph{size} condition on the jumps and a simple \emph{growth} condition on the sequence $t_n$. But in general, for the SLLN to hold, we neither need $(1+2/d)$-moments, nor does the sequence have to grow fast.

We write $a_n \sim b_n$ if $\lim_{n \to \infty} a_n / b_n = 1$ and 
$a_n \approx b_n$ if there are constants $C_1,C_2\in(0,\infty)$ such 
that $a_n / b_n  \in [C_1, C_2]$ for large values of $n$. The same notation is also used 
for continuous variables.

\bex[Condition \eqref{eq:f} is not necessary]\label{ex:1}
Neither the condition on the jumps nor the condition on $t_n$ in \eqref{eq:f} is necessary.
\benu
\item Consider the sequence $t_n=n^p$ with $p>0$ and $f(t)=t$. Then 
$\Delta t_n\sim p n^{p-1}$. Furthermore,
\begin{align*}
\int_0^\infty \lp \lp\frac{z}{n^p}\rp^{2/d}\wedge n^{p-1} \rp\frac{z}{n^p}\,\la(\dd z) 
=
\int_0^{n^{p+(p-1)d/2}} \frac{z^{1+2/d}}{n^{p(1+2/d)}}\,\la(\dd z) + 
\int_{n^{p+(p-1)d/2}}^\infty \frac{z}{n}\,\la(\dd z),
\end{align*}
where we use the convention $\int_a^b = \int_{(a,b]}$. 
Summing over $n\in\N$ and changing integral and summation, we obtain
\beq\label{eq:exp-2}
\int_0^\infty \left(
\sum_{n^{p + (p-1)d /2} \geq z} \frac{z^{1+2/d}}{n^{p(1+2/d)}}+ 
\sum_{n^{p + (p-1)d /2} < z} \frac{z}{n}\right)\,\la(\dd z). 
\eeq
The second sum in the integral above is finite only if 
$p + (p-1) d / 2 > 0$, or equivalently $p > d/(d+2)$.
Then the first sum in the integral is  
$\approx z^{1+2/d}z^{(1-p(1+2/d))/(p+(p-1)d/2)}=z$, while the second sum is 
$\approx z\log^+(z)\bone_{z>1}$. Hence, the expression in 
\eqref{eq:exp-2} is finite if and only if $\int_0^\infty z\log^+(z)\,\la(\dd z)<\infty$, 
which is always true by \eqref{eq:cond-lambda}. The same argument obviously applies to the 
negative jumps as well. Thus, under \eqref{eq:cond-lambda}, and if $\la((0,\infty))>0$ 
(resp., $\la((-\infty,0))>0$), the series in \eqref{eq:t-cond-0} (resp.,
\eqref{eq:t-cond-1}) is finite for the sequence $t_n=n^p$ if and only if $p>d/(d+2)$. This 
shows that the first part in \eqref{eq:f} is not a necessary condition.

\item
An easy counterexample also shows that the second condition in \eqref{eq:f} is not 
necessary. Consider the sequence $(t_n)_{n\in\bbn}$ that visits each $n\in\N$ exactly $n$ 
times, that is,
\beq\label{eq:seq-ex} t_1 = 1,\quad t_2=t_3=2,\quad t_4=t_5=t_6=3,\quad 
t_7=t_8=t_9=t_{10}=4\quad \text{etc.} \eeq
Then, from the SLLN on the sequence $n$, we derive
\[ \lim_{n\to\infty} \frac{Y_0(t_n)}{t_n} = \lim_{n\to\infty} \frac{Y_0(n)}{n} = m. \]
But we have
\[ \sum_{n=1}^\infty \frac{1}{t_n^{1+2/d}} = \sum_{n=1}^\infty \frac{n}{n^{1+2/d}} = 
\sum_{n=1}^\infty \frac{1}{n^{2/d}}, \]
which is infinite for $d\geq2$.
\eenu
\eex

In \eqref{eq:seq-ex}, we have seen a sequence $t_n$ on which the SLLN holds although it 
grows relatively slowly. Indeed, we have $t_n=O(\sqrt{n})$ and the SLLN fails on the 
sequence $\sqrt{n}$ if $d\geq2$; see the first part in Example~\ref{ex:1}. So the growth 
of a sequence $t_n$ does not fully determine whether the SLLN holds or not. In fact, we 
have found an example of two sequences $(s_n)_{n\in\N}$ and $(t_n)_{n\in\N}$ where we have 
$s_n\geq t_n$ for all $n\in\N$, but the SLLN only holds on $(t_n)_{n\in\N}$ and not on 
$(s_n)_{n\in\N}$. Hence, in order to determine whether the SLLN holds or not on a given 
sequence, we have to take into account its \emph{clustering} behavior, in addition to its 
speed. This is why the increments $\Delta t_n$ enter the conditions \eqref{eq:t-cond-0} 
and \eqref{eq:t-cond-1}. The following criterion is an improvement of 
Corollary~\ref{cor:1}.

\bcor\label{cor:2}
Suppose that
\beq\label{eq:f-2} \int_0^\infty |z|^{1+2/d}\,\la(\dd z)<\infty \qquad\lp \text{resp.,}\quad\int_{-\infty}^0 |z|^{1+2/d}\,\la(\dd z)<\infty\rp\eeq
and that $\la((0,\infty))\neq0$ (resp., $\la((-\infty,0))\neq0$). Then the series in \eqref{eq:t-cond-0} (resp., \eqref{eq:t-cond-1}) converges if and only if 
\beq\label{eq:f-3} \sum_{n=1}^\infty \frac{f(t_n)^{-2/d}\wedge \Delta t_n}{f(t_n)} < \infty.\eeq 

In particular, if $m_\la(1+2/d)<\infty$ and $\la\neq 0$, the SLLN holds on $t_n$ if and only if $t_n$ satisfies \eqref{eq:f-3} with $f(t)=t$.
\ecor
\bpr
We write the left-hand side of \eqref{eq:t-cond-0} as
\begin{align*}
&\sum_{n=1}^\infty \lp \int_0^{f(t_n)(\Delta t_n)^{d/2}} \frac{z^{1+2/d}}{f(t_n)^{1+2/d}}\,\la(\dd z) + \int_{f(t_n)(\Delta t_n)^{d/2}}^\infty \frac{z\Delta t_n}{f(t_n)}\,\la(\dd z) \rp,
\end{align*}
and split this sum into two parts, $I_1$ and $I_2$, according to whether $n$ belongs to
\[ \cali_1 = \{n\in\N: \Delta t_n \leq f(t_n)^{-2/d}\}\qquad \text{or}\qquad \cali_2 = \{n\in\N: \Delta t_n > f(t_n)^{-2/d}\}. \]
Then by \eqref{eq:cond-lambda} and \eqref{eq:f-3}, 
\begin{align*}
I_1 \leq \sum_{n\in \cali_1} \lp \frac{(f(t_n)(\Delta t_n)^{d/2})^{2/d}}{f(t_n)^{1+2/d}}\int_0^1 z \,\la(\dd z) + \frac{\Delta t_n}{f(t_n)} \int_0^\infty z\,\la(\dd z) \rp \leq 2\int_0^\infty z\,\la(\dd z) \sum_{n\in\cali_1} \frac{\Delta t_n}{f(t_n)}<\infty,
\end{align*}
and
\begin{align*}
I_2 &\leq \sum_{n\in \cali_2} \lp \frac{1}{f(t_n)^{1+2/d}}\int_0^\infty z^{1+2/d} \,\la(\dd z) +  \int_0^\infty z\frac{(z/f(t_n))^{2/d}}{f(t_n)}\,\la(\dd z) \rp\\
&\leq 2\int_0^\infty z^{1+2/d}\,\la(\dd z) \sum_{n\in\cali_2} \frac{1}{f(t_n)^{1+2/d}}<\infty,
\end{align*}
which shows one direction in \eqref{eq:f-3}. The other direction follows from
\begin{align*} I_1&\geq \sum_{n\in\cali_1} \lp \int_0^{f(t_n)(\Delta t_n)^{d/2}} \frac{z^{1+2/d}\Delta t_n}{f(t_n)}\,\la(\dd z) + \int_{f(t_n)(\Delta t_n)^{d/2}}^\infty \frac{z\Delta t_n}{f(t_n)}\,\la(\dd z) \rp\\
 &\geq \sum_{n\in\cali_1} \frac{\Delta t_n}{f(t_n)} \int_0^\infty (z\wedge z^{1+2/d})\,\la(\dd z)
\end{align*}
and
\begin{align*}
I_2&\geq \sum_{n\in\cali_2} \lp \int_0^{f(t_n)(\Delta t_n)^{d/2}} \frac{z^{1+2/d}}{f(t_n)^{1+2/d}}\,\la(\dd z) + \int_{f(t_n)(\Delta t_n)^{d/2}}^\infty \frac{z}{f(t_n)^{1+2/d}}\,\la(\dd z) \rp\\
&\geq \sum_{n\in\cali_2} \frac{1}{f(t_n)^{1+2/d}} \int_0^\infty (z\wedge z^{1+2/d})\,\la(\dd z).
\end{align*}
An analogous argument applies to \eqref{eq:t-cond-1}.
\epr

Is it possible to separate \eqref{eq:t-cond-0} and \eqref{eq:t-cond-1} into a condition on the Lévy measure $\la$ and a condition on the sequence $t_n$? And is it possible to determine whether the SLLN holds or not  by only looking at the sequence $t_n$, \emph{without} assuming a finite $(1+2/d)$-moment as in Corollary~\ref{cor:2}, but only assuming a finite first moment (or a finite $(1+\eps)$-moment as in \eqref{eq:cond-lambda})? The answer is \emph{no}, in both cases.

\bex[Conditions \eqref{eq:t-cond-0} and \eqref{eq:t-cond-1} are not separable]\label{ex:2}
Consider $f(t)=t$ and the sequence $t_n = (n(\log^+ n)^{1+\theta})^{d/(d+2)}$ for some $\theta>0$ (in particular, \eqref{eq:f-3} is satisfied). By the mean-value theorem, it is not difficult to see that 
\beq\label{eq:incr} \Delta t_n \approx n^{-2/(d+2)}(\log^+ n)^{(1+\theta)d/(d+2)},\quad \frac{\Delta t_n}{t_n} \approx n^{-1},\quad t_n(\Delta t_n)^{d/2}\approx (\log^+ n)^{(1+\theta)d/2}. \eeq

Next, let us take the L\'evy measure
\[ \la(\dd z) = z^{-1-\al}\bone_{[1,\infty)}(z)\,\dd z \]
for some $\al\in(1,1+2/d)$. In particular, if $\al\in(1,2)$, the Lévy noise will have the same jumps of size larger than $1$ as an $\al$-stable noise. It is easy to verify that the chosen Lévy measure has finite moments up to order $\al$ (but not including $\al$) so that \eqref{eq:cond-lambda} is satisfied, but \eqref{eq:f-2} is not. 

In this set-up, the series in \eqref{eq:t-cond-0} becomes
\beq\label{eq:series} \sum_{n=1}^\infty  \int_1^{t_n(\Delta t_n)^{d/2}} \lp \frac{z}{t_n}\rp^{1+2/d} z^{-1-\al} \,\dd z + \sum_{n=1}^\infty  \int_{t_n(\Delta t_n)^{d/2}}^\infty \frac{z\Delta t_n}{t_n} z^{-1-\al} \,\dd z. \eeq
Regarding the first sum, \eqref{eq:incr} implies that
\begin{align*} \int_1^{t_n(\Delta t_n)^{d/2}} \lp \frac{z}{t_n}\rp^{1+2/d} z^{-1-\al} \,\dd z&\approx \frac{1}{n(\log^+n)^{1+\theta}}\int_1^{(\log^+ n)^{(1+\theta)d/2}} z^{2/d-\al}\,\dd z \\
&\approx \frac{(\log^+ n)^{(1+2/d-\al)(1+\theta)d/2}}{n(\log^+n)^{1+\theta}},\end{align*}
which in turn shows that the first sum in \eqref{eq:series} converges if and only if 
\beq\label{eq:cond-log} \al>1+\frac{2}{d(1+\theta)}.  \eeq
The same holds for the second sum in \eqref{eq:series} because by \eqref{eq:incr},
\begin{align*}
\int_{t_n(\Delta t_n)^{d/2}}^\infty \frac{z\Delta t_n}{t_n} z^{-1-\al} \,\dd z \approx n^{-1} \int_{(\log^+ n)^{(1+\theta)d/2}}^\infty z^{-\al}\,\dd z \approx n^{-1}(\log^+ n)^{(1-\al)(1+\theta)d/2}.
\end{align*}
Altogether, the series in \eqref{eq:t-cond-0} converges if and only if \eqref{eq:cond-log} holds, which involves $\al$ (a parameter of the noise) and $\theta$ (a parameter of $t_n$) at the same time.
\eex

Corollary~\ref{cor:2} and Example~\ref{ex:2} also show the following peculiar fact. If the jumps of $\La$ have a finite $(1+2/d)$-moment, then whether we have the SLLN on $(t_n)_{n\in\N}$ or not, only depends on this sequence itself; the details of the Lévy measure (i.e., the distribution of the jumps) do not matter. So  if $\La$ has both positive and negative jumps, we either have the SLLN on $t_n$, or we see peaks in both directions, in the sense that the limit superior/inferior of $Y_0(t_n)/t_n$ is $\pm \infty$. 

But for noises with an infinite $(1+2/d)$-moment, and again jumps of both signs, we may have a sequence $t_n$ on which the SLLN (with $f(t)=t$) only fails in one direction. That is, we see peaks, for instance, for the limit superior, but then have convergence to the mean for the limit inferior. By Theorem~\ref{thm:cont}, this does not happen in continuous time: if we have both positive and negative jumps, we see both positive and negative peaks.

\subsection{Multiplicative noise with bounded nonlinearity}
Without much additional effort, we can generalize
Theorems~\ref{thm:cont}, \ref{thm:WLLN}, and \ref{thm:Yf-ds} to the stochastic heat equation with a bounded multiplicative nonlinearity. Consider \eqref{SHE}
where $\si\colon\R\to\R$ is a globally Lipschitz function that is \emph{bounded} and 
\emph{bounded away from $0$}, that is, there are constants $k_1,k_2>0$ such that 
\beq\label{eq:si} k_1<\si(x)<k_2,\qquad x\in\R.\eeq
It is well known from \cite[Théorème~1.2.1]{SLB98} that if $\la$ satisfies 
\eqref{eq:int-lambda} (in particular, if \eqref{eq:cond-lambda} holds), and $\dot\La$ has 
no Gaussian part, then \eqref{SHE} has a unique mild solution $Y$, that is, there is a 
unique predictable process $Y$ that satisfies the integral equation
\[
Y(t,x) = \int_0^t \int_{\R^d} g(t-s, x- y) \sigma(Y(s,y)) \, \Lambda(\dd s, \dd y). 
\]
We continue to 
write $Y_0(t)=Y(t,0)$.

\begin{theorem}\label{thm:cs-g} Suppose that $f\colon (0,\infty) \to (0,\infty)$ is nondecreasing and that $\si\colon\R\to\R$ is Lipschitz continuous and satisfies \eqref{eq:si}. Furthermore, assume that \eqref{eq:cond-lambda} holds.
\benu
\item Part (1) and part (2) of Theorem~\ref{thm:cont} remain valid.
\item Part (3) of Theorem~\ref{thm:cont} remains valid if we replace \eqref{eq:liminf} by
\[
\limsup_{t \to \infty} \frac{Y_0(t)}{t} <\infty  \qquad \left(\text{resp.}\quad 
\liminf_{t \to \infty} \frac{Y_0(t)}{t} >-\infty\right)\quad\as
\]
\eenu
\end{theorem}

\bthm\label{thm:WLLN-g} Suppose that $\la$ and $p$ are as in Theorem~\ref{thm:WLLN} and that $\si$ is Lipschitz continuous and satisfies \eqref{eq:si}. Then, instead of \eqref{eq:WLLN}, we have
\[ \limsup_{t\to\infty} \bbe\lb\lv \frac{Y_0(t)}{t}\rv^p\rb < \infty. \]
\ethm

\begin{theorem}\label{thm:ds-g} Suppose that $f$, $(t_n)_{n\in\N}$, and $\la$ satisfy the same hypotheses as in Theorem~\ref{thm:Yf-ds}. Moreover, let $\si$ be Lipschitz continuous with the property \eqref{eq:si}.
\benu
\item Part (1) of Theorem~\ref{thm:Yf-ds} remains valid.
\item Part (2) of Theorem~\ref{thm:Yf-ds} remains valid if $\kappa<\infty$ in 
\eqref{eq:kappa}, and if \eqref{eq:aux} is replaced by
\[
\limsup_{n \to \infty} \frac{Y_0(t_n)}{f(t_n)}  <\infty\qquad \left(\text{resp.}\quad 
\liminf_{n \to \infty} \frac{Y_0(t_n)}{f(t_n)} >-\infty\right)\qquad\text{a.s.}\]
\eenu
\end{theorem}

\brem If \eqref{eq:si} is violated, then we no longer expect the last three results to hold. For example, if we consider the parabolic Anderson model with Lévy noise, that is, \eqref{SHE} with $\si(x)=x$ (and a nonzero initial condition), it is believed that the solution develops exponentially large peaks as time tends to infinity (\emph{multiplicative} intermittency). While the moments of the solution indeed grow exponentially fast as shown in \cite{Chong18}, its pathwise asymptotic behavior is yet unknown. Only for a related model, where the noise $\dot\La$  is obtained by averaging a standard Poisson noise over a unit ball in space (so that the resulting noise is white in time but has a smooth covariance in space), it was proved in \cite{Comets05} for delta initial conditions that $t\mapsto \int_{\R^d} Y(t,x)\,\dd x$ has exponential growth in $t$ almost surely.
\erem

\section{Proofs}\label{sect:proof}

In the following, we denote the points of the 
Poisson random measure $\mu$ by $(\tau_i, \eta_i, \zeta_i)_{i \in \N}$ 
and refer to $\tau_i$, $\eta_i$, and $\zeta_i$ as the \emph{jump time}, \emph{jump location}, and \emph{jump 
size}, respectively. 
For the fixed reference point $x=(0,\dots,0)$ and some given $t>0$, we shall say that 
$(\tau_i,\eta_i,\zeta_i)$ with $\tau_i\leq t$ is a 
\bit
\item \emph{recent} (resp., \emph{old}) jump if $t-\tau_i\leq 1$ (resp., $t-\tau_i>1$);
\item \emph{close} (resp., \emph{far}) jump if $|\eta_i|\leq 1$ (resp., $|\eta_i|>1$);
\item \emph{small} (resp., \emph{large}) jump if $|\zeta_i|\leq 1$ (resp., $|\zeta_i|>1$).
\eit

A key to the proofs below is to decompose $Y_0(t)$ into a contribution of the recent close 
jumps and a contribution by all other jumps. That is, by \eqref{eq:mild} and \eqref{eq:def-Lambda}, we have
\[
Y_0(t) = \int_0^t \int_{\R^d} g(t-s, y) \Lambda(\dd s, \dd y) 
= \sum_{\tau_i \leq t} g(t-\tau_i, \eta_i) \zeta_i + m_0 t =: Y_1(t) + Y_2(t),
\]
where
\beq\label{eq:Y1Y2} 
Y_1(t)=\sum_{\tau_i \in (t-1,t],\, |\eta_i| \leq 1 } g(t - \tau_i, \eta_i)\zeta_i,\qquad 
Y_2(t) = m_0 t + \sum_{\tau_i \leq t-1\text{ or } |\eta_i| > 1 } g(t - \tau_i, 
\eta_i)\zeta_i.  
\eeq
We also consider the decomposition $Y_1 =Y^{+}_1+Y^{-}_1$, where $Y^{+}_1$ (resp., 
$Y^{-}_1$) only contains the positive (resp., negative) jumps in the definition of $Y_1$ 
in \eqref{eq:Y1Y2}.

\subsection{Some technical lemmas and notation}

We begin with four simple lemmas: a tail 
and a large deviation estimate for Poisson random variables, two elementary results  
for the heat kernel, and one simple result from analysis.

\begin{lemma} \label{lem:Poisson}
Suppose that $X_\lambda$ follows a Poisson distribution with parameter $\la>0$. 
\benu
\item For every $n\in\N$, we have
\[
\p(X_\la \geq n) \leq \frac{\la^n}{n!}.
\]
\item If $x \geq a \lambda$, where $a>1$, then
\begin{equation} \label{eq:poi-ineq-0}
\p ( X_\lambda \geq x ) \leq e^{-  x(\log a -1+\frac1 a)}. 
\end{equation}
In particular, for $a=2$, we have
\[
\p ( X_\lambda \geq x ) \leq e^{-  \delta_0 x},
\]
where $\delta_0 = \log 2 - 1/2 \approx 0.193$.
\eenu
\end{lemma}
\begin{proof}
	By Taylor's theorem, there exists $\theta\in[0,1]$ such that 
	\[
	\p(X_\la\geq n) = e^{-\la}\sum_{k=n}^\infty \frac{\la^k}{k!} = e^{-\la}\lp e^\la 
	-\sum_{k=0}^{n-1} \frac{\la^k}{k!} \rp = e^{-\la}\frac{e^{\theta \la}\la^n}{n!} 
\leq 
	\frac{\la^n}{n!},
	\]
	which proves the first claim. For the second claim, we have
	by Markov's inequality for any $u > 0$,
	\[
	\p(X_\lambda \geq x ) \leq  \p ( e^{u X_\lambda} \geq e^{ux} ) \leq 
	e^{-ux} \E [e^{u X_\lambda}] = e^{- (ux + \lambda - \lambda e^{u})}.
	\]
	Choosing $u= \log (x/\lambda)$, we obtain
	\[
	\p ( X_\lambda \geq x ) \leq e^{- x ( \log \frac{x}{\lambda} - 1 + 
\frac{\lambda}{x})}.
	\]
	Since $\lambda/x \in (0,1/a]$ and the function $u\mapsto \log u^{-1} - 1 + u$ is 
	decreasing on $(0,1)$, the statement follows.
\end{proof}

\begin{lemma}\label{lem:hk} Let $g$ be the heat kernel given in \eqref{eq:def-heat}.
\benu
\item For fixed $x \in \R^d$, 
$g(t,x)$ is increasing on $[0, |x|^2/(2d)]$, decreasing on $[|x|^2/(2d), \infty)$, and 
its maximum is
\begin{equation} \label{eq:gprop}
g(|x|^2/(2d), x) = \left( \frac{d}{2 \pi e} \right)^{d/2}   |x|^{-d} .
\end{equation}
\item The time derivative of $g$ is given by 
\beq\label{eq:gprime} 
\partial_t g(t,x)=\frac{e^{-\frac{|x|^2}{4 t}}
\left( \frac{\pi|x|^2}{t}-2\pi d \right)}
{(4 \pi t)^{d/2+1}},\qquad (t,x)\in (0,\infty)\times\R^d.  
\eeq
There exists a finite constant $C>0$ such that for all $(t,x)\in(0,\infty)\times\R^d$,
\beq \label{eq:gprimebound} 
|\partial_t g(t,x)|\leq 
\begin{cases} C|x|^{-d-2} &\text{if } |x|^2>2dt,\\
Ct^{-d/2-1}  &\text{if } |x|^2\leq 2dt.
\end{cases} 
\eeq
\eenu
\end{lemma}

\begin{proof}
Both \eqref{eq:gprop} and \eqref{eq:gprime} follow from standard analysis. The second 
half of \eqref{eq:gprimebound} follows easily from \eqref{eq:gprime}. For the first half, 
using the inequality $\sup_{y > 0} e^{-y} y^{d/2 + 2} < \infty$, we have for
$|x|^2 \geq 2dt$ that
\[
| \partial_t g(t,x) | \leq 
e^{-\frac{|x|^2}{4 t}}  \frac{\pi|x|^2}{t} (4 \pi t)^{-d/2-1}
= C
e^{-\frac{|x|^2}{4 t}} \left( \frac{|x|^2}{4 t} \right)^{d/2 +2} |x|^{-d-2}
\leq C |x|^{-d-2}.
\]
\end{proof}

\begin{lemma} \label{lem:heat-kernel-ineq}
For every $\varepsilon > 0$, there exists
$\delta=\delta(\varepsilon) >0$ such that
\beq\label{eq:ineq2}
g(t+ s, x) \geq (1 - \varepsilon ) g(t,x) 
\eeq
holds under either of the two following conditions:
\benu
\item  $s,t>0$, $x \in \R^d$, and $s/t \leq \delta$;
\item $s,t>0$, $x\in\R^d$, $|x|>1$, and $s<\delta$.
\eenu
\end{lemma}
\begin{proof}
By definition, we have $g(t+s,x)\geq (1-\eps)g(t,x)$ if and only if
\[
e^{\frac{|x|^2}{4} \left( \frac{1}{t} - \frac{1}{t+s} \right) }
\geq (1 - \varepsilon) \left( 1 + \frac{s}{t} \right)^{d/2}.
\]
The left-hand side is greater than or equal to 1, since the exponent is positive. 
Therefore, the inequality is true if the right-hand side is less than or equal to 1. 
Under condition (1), we have  
$s /t \in (0,\delta)$, so it suffices to make sure that
\[
(1 - \varepsilon) ( 1 + \delta )^{d/2} \leq 1.
\]
This holds if we choose
$\delta\leq \delta_1 = (1 - \varepsilon)^{-2/d} - 1$.

In case (2), if $t>1/(2d)$, we have $s/t<2d\delta$, so by (1), \eqref{eq:ineq2} holds if 
$s\leq \delta_2 = \delta_1/(2d)$. If $0<t\leq 1/(2d)$, then $|x|>1$ implies that
\[
\frac{g(t,x)}{g(t+s,x)} = \lp 1+\frac s 
t\rp^{d/2}e^{-\frac{|x|^2s}{4t(t+s)}} \leq \lp 1+\frac s 
t\rp^{d/2}e^{-\frac{s}{4t(t+s)}}. 
\]
For fixed $s>0$, elementary calculus shows that the term on the right-hand side 
reaches its unique maximum at
\[ t= \frac{\sqrt{d^2s^2+1}-ds+1}{2d} \geq \frac{1}{2d}. \]
Hence, for $t<1/(2d)$ and $s<\delta_2$, we obtain that 
\[ \frac{g(t,x)}{g(t+s,x)}  \leq \lp 1+2ds\rp^{d/2}e^{-\frac{d^2s}{2ds+1}}\leq 
(1+2d\delta_2)^{d/2} = (1-\eps)^{-1}.  \]
The claim now follows by taking $\delta(\eps)=\delta_1\wedge \delta_2 = \delta_2$.
\end{proof}

\blem\label{lem:simple}
Suppose that $f\colon (0,\infty)\to(0,\infty)$ is nondecreasing.
\benu
\item If $\int_1^\infty 1/f(t)\, \dd t < \infty$, then
\[
\lim_{t \to \infty} \frac{f(t)}{t} = \infty.
\]
\item If $\int_1^\infty 1/f(t) \,\dd t = \infty$, then 
\begin{equation} \label{eq:int1}
\int_1^\infty \frac{1}{f(t) \vee t}\, \dd t = \infty. 
\end{equation}
\eenu
\elem
\begin{proof}
	For (1), if we had
	$\liminf_{t \to \infty} f(t) / t  < \infty$, then there would be a sequence $(a_n)_{n\in\N}$ with $a_1=1$ and increasing to infinity such that $f(a_n) \leq Ca_n$ for all $n\in\N$. Then, because $f$ is nondecreasing,
	\[
	\int_1^\infty \frac{1}{f(t)} \,\dd t \geq \sum_{n=1}^\infty \int_{a_n}^{a_{n+1}} 
	\frac{1}{f(a_{n+1})}\, \dd t \geq \sum_{n=1}^\infty \frac{a_{n+1} - a_n}{C a_{n+1}} = \infty,
	\]
	which would be a contradiction, provided we can prove the last equality.
	
	For any fixed $m\in\N$, we have
	\[
	\sum_{n=m}^\infty \frac{a_{n+1} - a_n}{a_{n+1}} \geq \sum_{n=m}^{N-1} \frac{a_{n+1} - a_n}{a_{n+1}}\geq
	\sum_{n=m}^{N-1}  \frac{a_{n+1} - a_n}{a_{N}} = \frac{a_N - a_m}{a_N} \longrightarrow 1,
	\]
	as $N \to \infty$, so the infinite sum diverges as claimed.
	
	For (2), suppose that the integral in \eqref{eq:int1} were finite. As $f(t) \vee t$ 
	is nondecreasing, by the first part of the lemma,
	\[
	\lim_{t \to \infty} \frac{f(t) \vee t}{t} = \infty, 
	\]
	which would imply $f(t) \geq t$ for large values of $t$, and thus
	$\int_1^\infty 1/f(t) \,\dd t < \infty$, a contradiction.
\end{proof}

Finally, let us introduce some notation: $B(r) = \{ x \in \R^d\colon |x| \leq r \}$ 
denotes the ball with radius $r > 0$, and
$B(r_1, r_2) = \{ x \in \R^d\colon r_1 \leq |x| \leq r_2 \}$ for $0 < r_1 < r_2$.
Furthermore, $v_d$ is the volume of $B(1)$, and $\overline \lambda(x) = \lambda ( ( x, 
\infty))$. For $r > 0$ and $0 \leq r_1 < r_2$, we write
\beq\label{eq:def-C}
U(r) = \left\{ (y,z) \colon \frac{z}{|y|^d} > r,\, |y| \leq 1 \right\}, \qquad
U(r_1, r_2) = \left\{ (y,z) \colon \frac{z}{|y|^d} \in [r_1, r_2],\, |y| \leq 1 \right\}.
\eeq
If $\ell$ is the Lebesgue measure on $\R^d$, then short calculation gives
\begin{equation} \label{eq:C-vol}
\begin{split}
\ell \otimes \lambda ( U(r_1, r_2)) & = 
v_d \int_0^\infty \left( \frac{z}{r_1} \wedge 1 \right) - \left( \frac{z}{r_2} \wedge 1 
\right)\, \lambda(\dd z) \\
& = v_d \left[ ( r_1^{-1} - r_2^{-1}) \int_0^{r_1} z \,\lambda( \dd z) + 
\int_{r_1}^{r_2} \left( 1 - \frac{z}{r_2} \right) \,\lambda(\dd z) \right] \\
& =  \psi(r_1) - \psi(r_2),
\end{split}
\end{equation}
where 
\beq\label{eq:psi} 
\psi(r) = v_d\left( \frac1r \int_0^r z\,\la(\dd z) + \ov\la(r) \right) = 
\frac{v_d}{r}\int_0^r \ov\la(z)\,\dd z.
\eeq
We also use the simpler bound
\begin{equation} \label{eq:C-bound}
\ell \otimes \lambda ( U(r))  = v_d \int_0^\infty
\left( \frac{z}{r} \wedge 1  \right)\, \lambda(\dd z) 
\leq  \frac{v_d}{r} \int_0^\infty z \,\lambda(\dd z).
\end{equation}

\subsection{Recent close jumps}\label{sect:rc}

As mentioned in the introduction, and as can be seen from Figures \ref{fig} and \ref{fig2} and from 
Theorem \ref{thm:C}, the failure of the SLLN for $Y_0(t)$ is due to the recent close 
jumps, which we now examine in detail. 

We first
analyze the behavior of $Y_1(t)$ from \eqref{eq:Y1Y2} in continuous time and turn to the technically more involved setting in discrete time afterwards.


\begin{proposition}\label{prop:Y1-cont} Part (1) and (2) of Theorem~\ref{thm:cont} hold 
if $Y_0(t)$ is replaced by $Y_1(t)$ from \eqref{eq:Y1Y2}.
\end{proposition}

\begin{proof} It is enough to prove the statements when $\la((0,\infty))\neq 0$.
Let us first assume that   $\int_1^\infty 1/f(t) \,\dd t = \infty$.
Without loss of generality, we may assume that $f(t) \to \infty$. 
Introduce, for $K \geq1$ fixed, the events
\[
A_n = \left\{ \mu \left( [n, n+1] \times B( (Kf(n+2))^{-1/d}) \times (r,\infty) \right) 
\geq 1 \right\}, \qquad 
n \geq 0,
\]
where $r$ is chosen such that $\ov\la(r)=\la((r,\infty))>0$.
Then there is $C>0$, which is independent of $n$, such that
$\p ( A_n ) = 1 - e^{-v_d \ov\la(r) / (Kf(n+2)) } \geq C v_d\ov\la(r)/(K f(n+2))$,
and thus
\[
\sum_{n=1}^\infty \p ( A_n ) = \infty. 
\]
As the events $A_n$ are independent, the Borel--Cantelli lemma implies that $A_n$ occurs 
infinitely many times.

Recall that $Y_1^+(t)$ contains only the positive jumps in $Y_1$. If $n$ is large enough such that $(Kf(n+2))^{-2/d}/(2d)\leq1$, then
each time $A_n$ occurs, we have by \eqref{eq:gprop}, 
\[
\sup_{t \in [n, n+2]} Y^+_1(t) \geq \left(\frac{d}{2 \pi e} \right)^{d/2} K f(n+2)r,
\]
that is,
\[
\sup_{t \in [n, n+2]} \frac{Y_1^+(t)}{f(t)} \geq
\left(\frac{d}{2 \pi e} \right)^{d/2} K r.
\]
Since $A_n$ happens infinitely often and $K$ is arbitrarily large, we obtain
\beq\label{eq:help} \limsup_{t\to\infty}\frac{Y^+_1(t)}{f(t)} = \infty\quad \text{a.s.}\eeq
Let $t_n=t_n(\om)$ be a subsequence on which $Y^+_1(t_n(\om))(\om)/f(t_n(\om))\to\infty$ for almost all $\omega$.
Since $Y^+_1$ and $Y^-_1$ are independent, we can choose another sufficiently fast 
subsequence of $t_n(\omega)$, denoted by $t_{n_k}(\omega)=t_{n_k(\omega)}(\omega)$, 
on which $Y^-_1(t_{n_k}(\om))(\om)/f(t_{n_k}(\om)) \to 0$ as $k\to\infty$; see the argument after \eqref{eq:toshow} in
the proof of Proposition~\ref{prop:rec-close}. Hence, 
$Y_1(t_{n_k}(\om))(\om)/f(t_{n_k}(\om)) \to \infty$ as $k\to\infty$, which is the claim.

We now turn to the second part. Recalling the sets introduced in \eqref{eq:def-C}, we consider for $K \geq1$ the events
\begin{equation} \label{eq:def-AD}
\begin{split}
B_n & = \left\{ \mu \left( [n, n+1] \times U( f(n) / K ) \right) \geq 1 \right\}, \\
C_n & = \left\{ \mu \left( [n, n+1] \times U( n/(K\log n), f(n)/K ) \right) \geq 2 
\right\},\\
D_n & = \left\{ \mu \left( [n, n+1] \times U( 1,n/(K\log n) ) \right) \geq 6 \log n 
\right\},\\
E_{n, 0} & = \left\{ \mu \left( [n, n+1] \times U( 1/ \sqrt{n} , 1 ) \right) 
\geq 2 v_d  m_\lambda(1)\sqrt{n} \right\},\\
E_{n, j} & = \left\{ \mu \left( [n, n+1] \times 
U ( r_{j+1, n} , r_{j,n} ) \right) \geq 2 \Delta_{j, n} \right\}, \qquad n,j \geq 1,
\end{split}
\end{equation}
where the numbers $r_{j,n}$ and $\Delta_{j,n}$ are defined as
\beq\label{eq:rDelta1}
\begin{split}
& r_{1,n}=\frac{1}{\sqrt{n}},\qquad r_{j+1, n} =  \sup \{ r>0: \, 
\psi(r) \geq 
\psi(r_{j,n}) + 16 \log (j f(n)) \}, \\
& \Delta_{j,n} = \psi(r_{j+1,n})-\psi(r_{j,n})= 16 \log (jf(n)),\qquad n,j\geq1.
\end{split}
\eeq
Figure~\ref{table2} illustrates the partitioning induced by these sets.
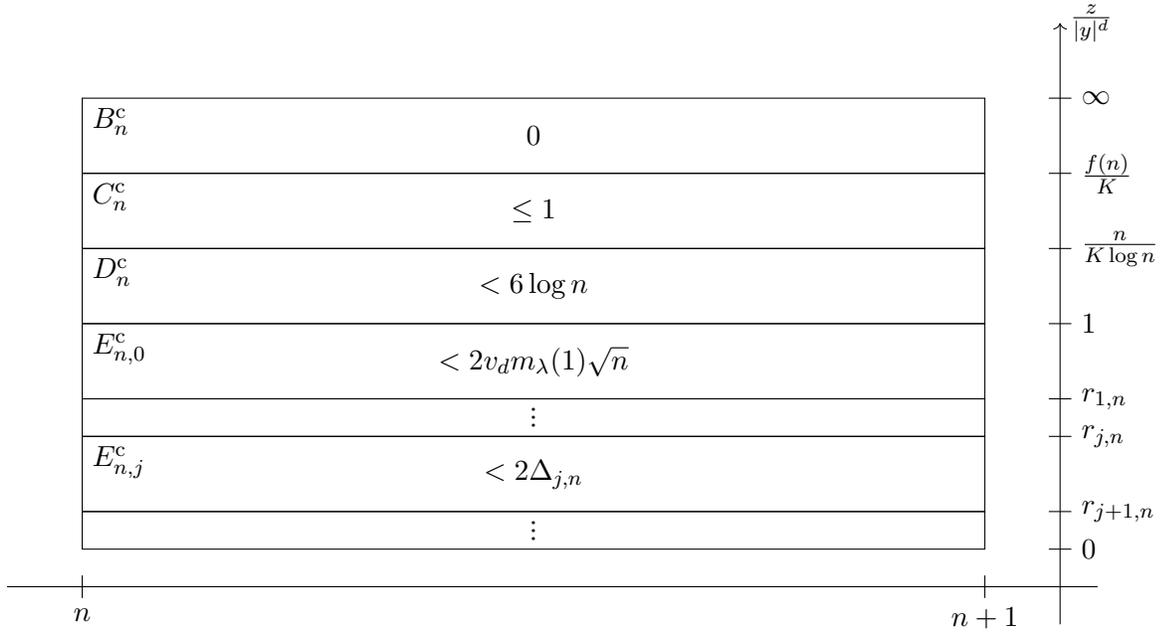
\begin{figure}
	\centering
	\begin{tikzpicture}
	\draw (0,0) rectangle node{$0$} (12,-1);
	\draw (0,-1) rectangle node{$\leq1$} (12,-2);
	\draw (0,-2) rectangle node{$<6\log n$} (12,-3);
	\draw (0,-3) rectangle node{$<2v_d m_\la(1)\sqrt{n}$} (12,-4);
	\draw (0,-4) rectangle node[rotate=-90]{...} (12,-4-0.5);
	\draw (0,-4-0.5) rectangle node{$<2\Delta_{j,n}$} (12,-5-0.5);
	\draw (0,-5-0.5) rectangle node[rotate=-90]{...} (12,-6);
	
	\draw[-] (13.5,-6.5) -- (-1,-6.5) ;
	\draw[->] (13,-7) -- (13,1) node[right]{$\frac{z}{|y|^d}$};
	
	\draw (12,-6.5+0.15) -- (12,-6.5-0.15) node[below] {$n+1$};
	\draw (0,-6.5+0.15) -- (0,-6.5-0.15) node[below] {$n$};
	
	\draw (13-0.15,0) -- (13+0.15,0) node[right] {$\infty$};
	\draw (13-0.15,-1) -- (13+0.15,-1) node[right] {$\frac{f(n)}{K}$};
	\draw (13-0.15,-2) -- (13+0.15,-2) node[right] {$\frac{n}{K\log n}$};
	\draw (13-0.15,-3) -- (13+0.15,-3) node[right] {$1$};
	\draw (13-0.15,-4) -- (13+0.15,-4) node[right] {$r_{1,n}$};
	\draw (13-0.15,-4.5) -- (13+0.15,-4.5) node[right] {$r_{j,n}$};
	\draw (13-0.15,-5.5) -- (13+0.15,-5.5) node[right] {$r_{j+1,n}$};
	\draw (13-0.15,-6) -- (13+0.15,-6) node[right] {$0$};
	
	\draw (0,0) node[below right] {$B_n^\Comp$};
	\draw (0,-1) node[below right] {$C_n^\Comp$};
	\draw (0,-2) node[below right] {$D_n^\Comp$};
	\draw (0,-3) node[below right] {$E_{n,0}^\Comp$};
	\draw (0,-4.5) node[below right] {$E_{n,j}^\Comp$};
	\end{tikzpicture}
	\caption{
		Restrictions on the maximal number of 
		jumps of $\mu$ if none of the events $B_n$--$E_{n,j}$ occur. The vertical direction $z/|y|^d$ indicates, up to a constant, the maximum peak size
		caused by a jump according to Lemma~\ref{lem:hk}. 
	}
	\label{table2}
\end{figure}

As the average of the decreasing function $\ov \la$, the function 
$\psi$ introduced in \eqref{eq:psi} is continuous and decreasing on $(0,\infty)$. 
Furthermore, since adding jumps to the Poisson random measure $\mu$ would only increase 
$Y_1(t)/t$, we may assume without loss of generality that 
$\la((0,\infty))=\infty$. In this case, we further have $\psi(0+)=\infty$, and $\psi$ is 
strictly decreasing. 
So for every $n\geq1$, the sequence $(r_{j,n})_{j\geq1}$ is strictly decreasing 
to $0$. 

Next, by Lemma~\ref{lem:Poisson} and the relations \eqref{eq:C-vol} and \eqref{eq:C-bound}, one can readily 
check that
\[\begin{split}
\p ( B_n ) &\leq v_d m_\la(1) \frac K {f(n)},\quad \p ( C_n ) \leq \frac{K^2v_d^2 
m_\la(1)^2(\log n)^2}{2 n^2},\quad \p ( D_n ) \leq n^{- 6 \delta_0}, \\
\p(E_{n,0}) &\leq e^{-v_d\delta_0m_\la(1)\sqrt{n}},\quad \p(E_{n,j}) \leq 
e^{-\delta_0\Delta_{j,n}}.
\end{split}\]
By the integrability assumption on $f$ and the fact that $6\delta_0>1$, all these probabilities are summable in $n$ and $j$, so 
the Borel--Cantelli lemma shows that almost surely, only finitely many of the events 
in \eqref{eq:def-AD} occur. Hence, if $n$ is large and $t \in [n, n+1]$, we have by 
\eqref{eq:gprop},
	\begin{align*} Y^+_1(t) &\leq \sum_{k = 0}^1 
	\sum_{\substack{\tau_i \in [n-k, n-k+1],  \\ \tau_i \leq t,\, |\eta_i| \leq 1 }}
	g(t-\tau_i, \eta_i)\zeta_i\\ 
	&\leq  2 \left(\frac{d}{2 \pi e} \right)^{d/2} \lp 
\frac{ f(n) + 6 n}{K} + 2v_dm_\la(1)\sqrt{n} + 2\sum_{j=1}^\infty r_{j,n}\Delta_{j,n} 
\rp.
	\end{align*}
The fact that $ \sum_{j=1}^\infty r_{j,n}\Delta_{j,n} < \infty$ is proved 
in a more general set-up at the end of the proof of Proposition \ref{prop:rec-close};
see \eqref{eq:sum-r-Delta}. Moreover, since $f$ is nondecreasing and $\int_0^\infty 1/f(t)\, 
\dd 
t < \infty$, we have
$t / f(t) \to 0$ as $t\to\infty$ by Lemma~\ref{lem:simple}. Thus, we see that
\[
\limsup_{t \to \infty} \frac{Y^+_1(t)}{f(t)} \leq \left(\frac{d}{2 \pi e} 
\right)^{d/2} \frac{2}{K}, 
\]
and we obtain the statement by letting $K \to \infty$. An analogous argument 
applies to $Y^-_1(t)$.
\end{proof}

For discrete subsequences, we need to refine the techniques applied in the proof of Proposition~\ref{prop:Y1-cont}.

\begin{proposition}\label{prop:rec-close}
Assume that
$(t_n)_{n\in\N}$ is a nondecreasing sequence tending to infinity. If $\la$ satisfies the 
second condition in \eqref{eq:cond-lambda}, the series in \eqref{eq:t-cond-0} (resp., 
\eqref{eq:t-cond-1}) is finite, and $f$ is unbounded, then
\[
\limsup_{n\to\infty} \frac{Y_1(t_n)}{f(t_n)} = 
0\qquad\left(\text{resp.}\quad \liminf_{n \to \infty} \frac{Y_1(t_n)}{f(t_n)} = 0 
\right)\qquad\text{a.s.}
\]
\end{proposition}

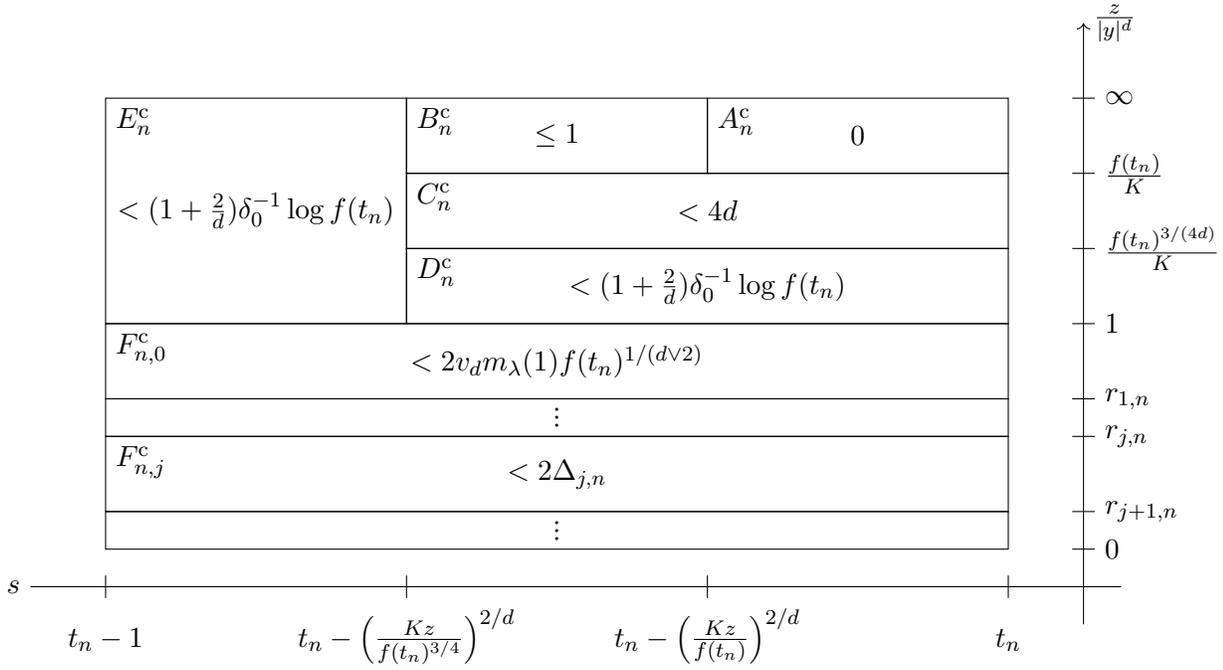
\begin{figure}[tb]
	\centering
	\begin{tikzpicture}
	\draw (0,0) rectangle node{$<(1+\frac2d)\delta_0^{-1}\log f(t_n)$} (4,-3);
	\draw (4,0) rectangle node{$\leq1$} (8,-1);
	\draw (8,0) rectangle node{$0$} (12,-1);
	\draw (4,-1) rectangle node{$<4d$} (12,-2);
	\draw (4,-2) rectangle node{$<(1+\frac2d)\delta_0^{-1}\log f(t_n)$} (12,-3);
	\draw (0,-3) rectangle node{$<2v_d m_\la(1)f(t_n)^{1/(d\vee2)}$} (12,-4);
	\draw (0,-4) rectangle node[rotate=-90]{...} (12,-4-0.5);
	\draw (0,-4-0.5) rectangle node{$<2\Delta_{j,n}$} (12,-5-0.5);
	\draw (0,-5-0.5) rectangle node[rotate=-90]{...} (12,-6);
	
	\draw[-] (13.5,-6.5) -- (-1,-6.5) node[left]{$s$};
	\draw[->] (13,-7) -- (13,1) node[right]{$\frac{z}{|y|^d}$};
	
	\draw (12,-6.5+0.15) -- (12,-6.5-0.15) node[below] {$t_n\vphantom{\lp\frac{Kz}{f(t_n)^{3/4}}\rp^{2/d}}$};
	\draw (8,-6.5+0.15) -- (8,-6.5-0.15) node[below] {$t_n-\lp\frac{Kz}{f(t_n)}\rp^{2/d}$};
	\draw (4,-6.5+0.15) -- (4,-6.5-0.15) node[below] {$t_n-\lp\frac{Kz}{f(t_n)^{3/4}}\rp^{2/d}$};
	\draw (0,-6.5+0.15) -- (0,-6.5-0.15) node[below] {$t_n-1\vphantom{\lp\frac{Kz}{f(t_n)^{3/4}}\rp^{2/d}}$};
	
	\draw (13-0.15,0) -- (13+0.15,0) node[right] {$\infty$};
	\draw (13-0.15,-1) -- (13+0.15,-1) node[right] {$\frac{f(t_n)}{K}$};
	\draw (13-0.15,-2) -- (13+0.15,-2) node[right] {$\frac{f(t_n)^{3/(4d)}}{K}$};
	\draw (13-0.15,-3) -- (13+0.15,-3) node[right] {$1$};
	\draw (13-0.15,-4) -- (13+0.15,-4) node[right] {$r_{1,n}$};
	\draw (13-0.15,-4.5) -- (13+0.15,-4.5) node[right] {$r_{j,n}$};
	\draw (13-0.15,-5.5) -- (13+0.15,-5.5) node[right] {$r_{j+1,n}$};
	\draw (13-0.15,-6) -- (13+0.15,-6) node[right] {$0$};
	
	\draw (0,0) node[below right] {$E_n^\Comp$};
	\draw (4,0) node[below right] {$B_n^\Comp$};
	\draw (8,0) node[below right] {$A_n^\Comp$};
	\draw (4,-1) node[below right] {$C_n^\Comp$};
	\draw (4,-2) node[below right] {$D_n^\Comp$};
	\draw (0,-3) node[below right] {$F_{n,0}^\Comp$};
	\draw (0,-4.5) node[below right] {$F_{n,j}^\Comp$};
	\end{tikzpicture}
	\caption{
		Restrictions on the maximal number of 
		jumps of $\mu$ if none of the events $A_n$--$F_{n,j}$ occur. The horizontal direction $s$ indicates how recent the jumps are relative to $t_n$. The $\wedge$-constraints in the $s$-coordinate, which make sure that only recent jumps are considered and that the sets $A_n$--$F_{n,j}$ are disjoint for different values to $t_n$, are not shown for the sake of clarity. 
	}
	\label{table}
\end{figure}

\begin{proof}
It suffices by symmetry to show 
the statement concerning the limit superior.  The claim now 
follows if we can show that the finiteness of the series in \eqref{eq:t-cond-0} implies
\beq\label{eq:toshow} \lim_{n\to\infty} \frac{Y^{+}_1(t_n)}{f(t_n)} =  0\quad\text{a.s.} 
\eeq 
Indeed, \eqref{eq:toshow} immediately gives $\limsup_{n\to\infty} {Y_1(t_n)}/{f(t_n)} \leq 
0$. For the other direction, observe that for a sufficiently fast subsequence 
$(t_{n_k})_{k\in\N}$ of $(t_n)_{n\in\N}$, the series in \eqref{eq:t-cond-1} will be 
finite because $f$ is unbounded, so \eqref{eq:toshow} and a symmetry argument prove that 
$Y^{-}_1(t_{n_k})/f(t_{n_k}) \to 0$ a.s.\ as $k\to\infty$. 
Together with \eqref{eq:toshow}, this means that $\limsup_{n\to\infty} {Y_1(t_n)}/{f(t_n)} 
\geq 0$. Actually, this argument shows that for \emph{any} $f$ increasing to infinity, and for 
\emph{any} $t_n$ increasing to infinity,
\begin{equation} \label{eq:liminfsup}
\limsup_{n \to \infty} \frac{Y_1^{-}(t_n)}{f(t_n)} = 0, \qquad
\liminf_{n \to \infty} \frac{Y_1^{+}(t_n)}{f(t_n)} = 0.
\end{equation}

In order to prove \eqref{eq:toshow}, it is no restriction to assume $\la((-\infty,0))=0$ 
and $m_0=0$, in which case we have $m_\lambda(1) = \int_0^\infty z \,\lambda(\dd z)$ and 
$Y^{+}_1 = Y_1$. Next, we redefine the numbers in \eqref{eq:rDelta1} by setting
\begin{equation} \begin{split}\label{eq:Delta-n}
r_{1,n}&=\frac{1}{f(t_n)^{1/(d\vee 2)}},\qquad r_{j+1, n} =  \sup \{ r>0: \, \psi(r) 
\geq 
\psi(r_{j,n}) + 16 \log (jf(t_n)) \},\\
\Delta_{j,n} &= \psi(r_{j+1,n})-\psi(r_{j,n})= 16 \log 
(jf(t_n)),\qquad n,j\geq1.
\end{split}
\end{equation}
As explained after \eqref{eq:rDelta1}, we may assume that for every $n\geq1$, the sequence $(r_{j,n})_{j\geq1}$ is strictly decreasing to $0$.

Now let $K>1$ be arbitrary but fixed for the moment. 
Recalling the definition of $\delta_0$ 
from Lemma~\ref{lem:Poisson}, we then consider the following events 
for $n,j\geq1$ (see also Figure~\ref{table} for a summary picture):
\allowdisplaybreaks[3]
\begin{align*}
	A_n&=\left\{ \mu\lp \left\{ (s,y,z): s\in\lb   t_n - \left(1\wedge \left( 
	\frac{Kz}{f(t_n)}\right)^{2/d}\wedge \Delta t_n\right),t_n\rb,\, (y,z)\in U\lp\frac 
	{f(t_n)}{K}\rp \right\} \rp \geq1 \right\},\\
	B_n&=\Bigg\{ \mu\Bigg( \Bigg\{ (s,y,z): s\in\lb  t_n - \left(1\wedge \left( 
		\frac{Kz}{f(t_n)^{3/4}}\right)^{2/d} \wedge \Delta t_n\right), t_n - \left(1\wedge \left( 
		\frac{Kz}{f(t_n)}\right)^{2/d}\wedge \Delta t_n\right)\rb,\\
		&\quad\qquad (y,z)\in U\lp\frac 
	{f(t_n)}{K}\rp \Bigg\} \Bigg) \geq 2 \Bigg\},\\
	C_n&=\Bigg\{ \mu\Bigg( \Bigg\{ (s,y,z): s\in\lb t_n - \left(1\wedge \left( 
	\frac{Kz}{f(t_n)^{3/4}}\right)^{2/d}\wedge \Delta t_n\right),t_n\rb,\\
	&\quad\qquad (y,z)\in U\left(\frac{f(t_n)^{3/(4d)}}{K},\frac 
	{f(t_n)}{K}\right) \Bigg\} \Bigg) \geq 4d \Bigg\},\\
	D_n&=\Bigg\{ \mu\lp \left\{ (s,y,z): s\in\lb t_n - \left(1\wedge \left( 
	\frac{Kz}{f(t_n)^{3/4}}\right)^{2/d}\wedge \Delta t_n\right),t_n\rb,\,(y,z)\in U\left(1,\frac 
	{f(t_n)^{3/(4d)}}{K}\right) \right\} \rp\\
	&\quad\qquad \geq (1+\textstyle \frac2d)\delta_0^{-1} \log f(t_n) \Bigg\},\\
	E_n&=\Bigg\{ \mu\lp \left\{ (s,y,z): s\in\lb  t_n - (1 \wedge \Delta t_n), 
t_n - \left(1\wedge \left( 
	\frac{Kz}{f(t_n)^{3/4}}\right)^{2/d}\wedge \Delta t_n\right)\rb,\,(y,z)\in U(1) \right\} 
	\rp\\
	&\quad\qquad \geq (1+\textstyle \frac2d)\delta_0^{-1} \log f(t_n) \Bigg\},\\
	F_{n,0}&=\Big\{ \mu\lp \left\{ (s,y,z): s\in [  t_n - (1\wedge \Delta t_n), t_n 
	],\,(y,z)\in U( r_{1,n},1) \right\} \rp \geq 2v_dm_\la(1)f(t_n)^{1/(d\vee 2)} \Big\},\\
	F_{n,j}&=\Big\{ \mu\lp \left\{ (s,y,z): s\in [  t_n - (1\wedge \Delta t_n), t_n 
	],\,(y,z)\in U( r_{j+1,n}, r_{j,n}) \right\} \rp \geq 2\Delta_{j,n} \Big\}.
\end{align*}


Each $\mu(\cdot)$ variable follows a Poisson distribution, so we can use 
Lemma~\ref{lem:Poisson} to estimate the probabilities of these events. Writing 
$$p_n=\int_0^\infty \lp \lp \frac{z}{f(t_n)}\rp^{2/d}\wedge \Delta t_n\rp 
\frac{z}{f(t_n)}\,\la(\dd z),$$ 
we obtain 
\beq\label{eq:A-C-prob}
\begin{split}
\p(A_n) & \leq v_d\int_0^\infty \lp \lp \frac{Kz}{f(t_n)}\rp^{2/d}\wedge \Delta 
t_n\wedge 1\rp 
\frac{Kz}{f(t_n)}\,\la(\dd z)\leq v_d K^{1+2/d} p_n,\\
\p(B_n)&\leq \frac{(K^{1+2/d}v_d)^2}{2}\lp \int_0^\infty  \lp
\lp\frac{z}{f(t_n)^{3/4}}\rp^{2/d}\wedge \Delta t_n \wedge 1\rp 
\frac{z}{f(t_n)}\,\la(\dd z) \rp^2 \\
&\leq C\lp \int_0^\infty  \lp
\lp\frac{z}{f(t_n)}\rp^{2/d}\wedge \Delta t_n \wedge 1\rp 
\frac{z}{f(t_n)^{1-1/(2d)}}\,\la(\dd z) \rp^2\\
& \leq C \int_0^\infty  \lp
\lp\frac{z}{f(t_n)}\rp^{2/d}\wedge \Delta t_n\rp 
\frac{z}{f(t_n)^{1-1/(2d)}}\,\la(\dd z) \frac{m_\la(1)}{f(t_n)^{1-1/(2d)}}\leq Cp_n,\\
\p(C_n)&\leq \frac{(K^{1+2/d}v_d)^{4d}}{(4d)!}\lp \int_0^\infty  \lp
\lp\frac{z}{f(t_n)^{3/4}}\rp^{2/d}\wedge \Delta t_n \wedge 1\rp 
\frac{z}{f(t_n)^{3/(4d)}}\,\la(\dd z) \rp^{4d}\\
&\leq C\lp \int_0^\infty  \lp
\lp\frac{z}{f(t_n)}\rp^{2/d}\wedge \Delta t_n \wedge 1\rp 
\frac{z}{f(t_n)^{1/(4d)}}\,\la(\dd z) \rp^{4d}\\
& \leq C \int_0^\infty  \lp
\lp\frac{z}{f(t_n)}\rp^{2/d}\wedge \Delta t_n\rp 
\frac{z}{f(t_n)^{1/(4d)}}\,\la(\dd z) \frac{m_\la(1)^{4d-1}}{f(t_n)^{1-1/(4d)}}\leq Cp_n. 
\end{split}\eeq

Next, we make the following observation: on the one hand, if $\Delta t_n\geq 
(R/f(t_n))^{2/d}$, where $R\in(0,\infty)$ is such that $\la((0,R])\neq 0 \neq 
\la([R,\infty))$, then 
\beq\label{eq:pn1} p_n \geq \int_0^{f(t_n)(\Delta t_n)^{d/2}} 
\frac{z^{1+2/d}}{f(t_n)^{1+2/d}} \,\la(\dd z) \geq \frac{1}{f(t_n)^{1+2/d}} \int_0^R 
z^{1+2/d}\,\la(\dd z)\geq \frac{C}{f(t_n)^{1+2/d}};\eeq
on the other hand, if $\Delta t_n< (R/f(t_n))^{2/d}$, then
\beq\label{eq:pn2} p_n\geq \frac{\Delta t_n}{f(t_n)}\int_{f(t_n)(\Delta t_n)^{d/2}}^\infty 
z\,\la(\dd z) \geq C\frac{\Delta t_n}{f(t_n)}. \eeq
As a consequence, upon noticing from \eqref{eq:C-bound} that the intensity of the Poisson 
variables in the definition of $D_n$ and $E_n$ is bounded by $v_d m_\la(1)(\Delta 
t_n\wedge 1)$, we obtain from both parts of Lemma~\ref{lem:Poisson} that for large values 
of $n$,
\beq\label{eq:D-E-prob}
\p(D_n) \vee \p(E_n)\leq \begin{cases} e^{-\delta_0(1+2/d)\delta_0^{-1} \log f(t_n)} = 
\displaystyle \frac{1}{f(t_n)^{1+2/d}}\leq Cp_n &\text{if } \Delta t_n\geq 
(R/f(t_n))^{2/d},\\ \displaystyle \frac{(v_d m_\la(1)(\Delta t_n\wedge 1))^{\lceil 1+\frac 
d 2\rceil}}{\lceil 1+\frac 
d 2\rceil!} \leq C\frac{\Delta t_n}{f(t_n)}\leq Cp_n&\text{if } \Delta 
t_n< (R/f(t_n))^{2/d}.\end{cases} 
\eeq

For the sets $F_{n,j}$, we distinguish between the same two cases as in 
\eqref{eq:D-E-prob}. Since  \eqref{eq:C-vol} and \eqref{eq:C-bound} imply that the 
intensity of the respective Poisson variable is bounded by $(1\wedge \Delta t_n)v_d 
m_\la(1)r_{1,n}^{-1} = 
(1\wedge \Delta t_n)v_d m_\la(1)f(t_n)^{1/(d \vee 2)}$ for $j=0$, and by $(1\wedge\Delta 
t_n)(\psi(r_{j+1,n})-\psi(r_{j,n})) = 
(1\wedge \Delta t_n)\Delta_{j,n}$ for $j\geq1$, we obtain from Lemma~\ref{lem:Poisson} and 
the inequalities \eqref{eq:pn1} and \eqref{eq:pn2}, 
\beq\label{eq:F0-prob} \p(F_{n,0})\leq \begin{cases} 
e^{-2\delta_0v_dm_\la(1)f(t_n)^{1/(d\vee 2)}} \leq \displaystyle \frac{C}{f(t_n)^{1+2/d}} 
\leq Cp_n &\text{if } \Delta t_n\geq (R/f(t_n))^{2/d}, \\ \displaystyle\frac{((1\wedge 
\Delta t_n)v_d m_\la(1)f(t_n)^{1/(d\vee2)})^{2+d}}{(2+d)!}\leq C\frac{\Delta 
t_n}{f(t_n)}\leq Cp_n &\text{if } \Delta t_n< (R/f(t_n))^{2/d}. \end{cases}\eeq
Similarly, using the relation $16\delta_0 = 3.090\ldots \geq 3\geq 1+2/d$, we deduce for 
$\Delta t_n\geq (R/f(t_n))^{2/d}$,
\beq\label{eq:Fj-prob-1}
\p(F_{n,j}) \leq  e^{-\delta_0 \Delta_{j,n}} = (jf(t_n))^{-16\delta_0} \leq 
j^{-16\delta_0}p_n.
\eeq
For $\Delta t_n< 
(R/f(t_n))^{2/d}$, using \eqref{eq:poi-ineq-0} with $a=2/\Delta t_n$, the inequality $\log x -1 +1/x\geq 
e^{-1}\log x$ for $x\geq e$, and the fact that $2/\Delta t_n\geq 2(f(t_n)/R)^{2/d} \geq 
e$ and $\Delta_{j,n}\geq e$ for all but finitely many $n$, we obtain 
\beq\label{eq:Fj-prob-2}
\begin{split}
\p(F_{n,j}) &\leq e^{ -2  \Delta_{j,n} 
( \log \frac{2}{\Delta t_n} - 1 + \frac{\Delta t_n}{2})} 
\leq  e^{ - \frac{32}{e} \log (jf(t_n)) \log \frac{2}{\Delta t_n}}  \leq e^{ - \frac{32}{e} \left( \log (jf(t_n)) + \log \frac{2}{\Delta t_n} 
\right) }\\
&
= j^{-\frac{32}{e}} f(t_n)^{-\frac{32}{e}} \left( \frac{\Delta t_n}{2} \right)^{\frac{32}{e}}  \leq C \frac{\Delta t_n}{f(t_n)} j^{-\frac{32}{e}} \leq Cj^{-\frac{32}{e}} 
p_n.
\end{split}
\eeq

Altogether, \eqref{eq:A-C-prob}, \eqref{eq:D-E-prob}, \eqref{eq:F0-prob}, 
\eqref{eq:Fj-prob-1}, and \eqref{eq:Fj-prob-2} show that
\[
\sum_{n=1}^\infty \p\lp A_n\cup B_n\cup C_n\cup D_n\cup E_n \cup 
F_{n,0}\cup \bigcup_{j=1}^\infty F_{n,j}\rp < \infty.
\]
Thus, the Borel--Cantelli lemma implies that only finitely many of these events occur. 

Suppose now that $n_0=n_0(\omega,K)\in\N$ is large enough such that none of these events 
happens for $n\geq n_0$. In particular, there is no jump as described in $A_n$, and fewer 
than $4d$ jumps as in $C_n$. By \eqref{eq:gprop}, each of these jumps 
	contributes to $Y_1(t_n)$ by a term bounded by $(d/(2\pi 
	e))^{d/2} f(t_n)/K$. With the same reasoning, we bound the maximum contribution of a jump as 
	described in the sets $D_n$ and $(F_{n,j})_{j\geq0}$, whereas we use the simple estimate 
	$g(t-\tau_i,\eta_i)\zeta_i \leq (4\pi(t-\tau_i))^{-d/2}\zeta_i$ for those jumps 
	$(\tau_i,\eta_i,\zeta_i)$ that are described in the sets $B_n$ and $E_n$. Hence, we obtain for $n\geq n_0$,
	\begin{align*}
	Y_1(t_n) &\leq (4\pi)^{-d/2} 
	\frac{f(t_n)}{K}+4d \lp \frac{d}{2\pi e}\rp^{d/2} \frac{f(t_n)}{K} + \lp 1+\frac 2 d\rp \delta_0^{-1} \log f(t_n) \frac{f(t_n)^{3/(4d)}}{K}\\
	 &\quad+ \lp 1+\frac 2 d\rp \delta_0^{-1} \log f(t_n) (4\pi)^{-d/2} \frac{f(t_n)^{3/4}}{K}+
	2v_d m_\la(1) f(t_n)^{1/(d\vee 2)}\lp \frac{d}{2\pi e}\rp^{d/2} \\
	&\quad+ 2 \lp\frac{d}{2\pi e}\rp^{d/2}\sum_{j=1}^\infty r_{j,n}\Delta_{j,n}.
	\end{align*}
	
	Since $K$ can be taken arbitrarily large,  \eqref{eq:toshow} follows when we show that
	\beq\label{eq:sum-r-Delta} \sum_{j=1}^\infty r_{j,n}\Delta_{j,n} = O(1) 
	\eeq
	as $n\to\infty$. To this end, observe from \eqref{eq:Delta-n} that if $n$ is large enough such that 
	$f(t_n)\geq 8 \geq e^2$, then 
	\begin{align*}
	\psi(r_{j+1, n}) &= \psi(r_{1,n}) +16 \sum_{k=1}^j  \log (k f(t_n)) \geq 16\int_1^j \log 
	x\,\dd x +16j\log f(t_n)\\ &= 16(j\log(jf(t_n))-j+1) \geq 8 j \log(jf(t_n))
	\end{align*}
	for all $j\geq1$,
	from which we deduce
	\[
	r_{j+1,n } \leq \psi^{-1} ( 8 j \log (jf(t_n)) ),
	\]
	where $\psi^{-1}$ is the inverse function of $\psi$ (recall that $\psi$ is assumed to be strictly decreasing). 
	Since  $8j\log (jf(t_n))-8(j-1)\log((j-1)f(t_n))\geq 8\frac{\log 2}{\log 3}\log((j+1)f(t_n))$ for 
	all $j\geq2$ and $8\log f(t_n)-8\log 2\geq 4 \log(2f(t_n))$ for $f(t_n)\geq 8$, we have by 
	Riemann-sum approximation,
	\begin{align*}
	\sum_{j=1}^\infty r_{j,n} \Delta_{j, n}  &\leq 
	16 \frac{\log f(t_n)}{f(t_n)^{1/(d\vee2)}} + 16 \sum_{j=1}^\infty \log ((j+1)f(t_n)) \psi^{-1}(8 j \log 
	(jf(t_n)) )\\
	&\leq o(1)+ 2\frac{\log 3}{\log 2} \int_{8\log 2}^\infty \psi^{-1}(x)\, \dd x.
	\end{align*}
	Note that by \eqref{eq:psi} and a change of variable,
	\[ \int_{\psi(1)}^\infty \psi^{-1}(x)\,\dd x = -\int_{0}^{1} y\,\psi(\dd y) = v_d 
	\int_0^1 z|\log z|\,\la(\dd z), \]
	which is assumed to be finite.
	Therefore, \eqref{eq:sum-r-Delta} holds and the proposition is proved.
\end{proof}

Next, we show the converse statement.

\begin{proposition} \label{prop:ssconverse} 
	Let 
	$f\colon(0,\infty)\to(0,\infty)$ be a nondecreasing function and $t_n$ be a 
	nondecreasing sequence tending to infinity. Assume that $\la$ satisfies \eqref{eq:cond-lambda}. If \eqref{eq:t-cond-0} (resp., \eqref{eq:t-cond-1}) holds, then
	\[
	\limsup_{n \to \infty} \frac{Y_1(t_n)}{f(t_n)} = \infty \qquad\left(\text{resp.,}\quad \liminf_{n \to \infty} \frac{Y_1(t_n)}{f(t_n)} = -\infty \right) \qquad\text{a.s.}
	\]
\end{proposition}

\begin{proof} By symmetry, it suffices to show the assertion under \eqref{eq:t-cond-0}.
Clearly, we may assume that $f$ converges to $\infty$, otherwise we can change $f$ to 
a larger function such that \eqref{eq:t-cond-0} still holds.

	Consider for $n\geq1$ and $K\geq1$ the events
	\[
	\begin{split}
	G_n &= \Bigg\{  \mu \Bigg(  \Bigg\{(s,y,z): s\in \left[ \left( t_n - \frac{2z^{2/d}}{(Kf(t_n))^{2/d}} \right) \vee 
	\left( t_{n-1} - \frac{z^{2/d}}{(Kf(t_{n-1}))^{2/d}} \right), t_n - \frac{z^{2/d}}{(Kf(t_n))^{2/d}}  
	\right],
	\\
	& \quad \qquad y\in B \left( \frac{z^{1/d}}{(Kf(t_n))^{1/d}}, 
	\frac{(2z)^{1/d}}{(Kf(t_n))^{1/d}} \right),~ z> 0 \Bigg)
	\geq 1 \Bigg\}.
	\end{split}
	\]
	Note that the second term of the maximum in the time 
	variable makes these events independent, and therefore we can apply the 
	second Borel--Cantelli lemma.
	Since
	\[
t_n - \frac{z^{2/d}}{(Kf(t_n))^{2/d}}   - \left(t_{n-1} - \frac{z^{2/d}}{(Kf(t_{n-1}))^{2/d}}\right) \geq \Delta t_n, 
	\]
	we have
	\[
	\p (G_n ) \geq \int_0^\infty  \left( \lp\frac{z}{K f(t_n)}\rp^{2/d}  \wedge  \Delta t_n 
	\right)\frac{v_d z}{K f(t_n)} \,\la(\dd z),
	\]
	which is not summable by assumption \eqref{eq:t-cond-0}. Thus, with probability $1$, 
	infinitely many of the events $G_n$ occur. If $G_n$ occurs, then there is at least one 
	jump, say $(\tau,\eta,\zeta)$, with $t_n-\tau\in 
	[(\zeta/(Kf(t_n)))^{2/d},2(\zeta/(Kf(t_n)))^{2/d}]$, $|\eta| \leq 
(2\zeta/(Kf(t_n)))^{1/d}$, and 
	$\zeta>0$. Therefore, from the definition of the heat kernel, we obtain in this case
	\[
	Y^+_1(t_n) \geq \frac{1}{(4\pi\times 2(\zeta/(Kf(t_n)))^{2/d} )^{d/2}} 
e^{-\frac{(2\zeta/(Kf(t_n)))^{2/d}}{4(\zeta/(Kf(t_n)))^{2/d}}}\zeta= 
\frac{e^{-2^{2/d-2}}}{(8\pi)^{d/2}} K f(t_n).
	\]
	As $K$ is as large as we want, we can extract for almost all $\omega$, a subsequence $t_{n_k(\omega)}$ such that
	\[	\lim_{k\to\infty} \frac{Y^+_1(t_{n_k(\omega)})(\omega)}{f(t_{n_k(\omega)})} = \infty.\]
The statement now follows in the same way as explained after \eqref{eq:toshow}.
\end{proof}

\subsection{Old jumps and far jumps}\label{sect:of}

The goal of this subsection is to prove that $Y_2(t)$ in \eqref{eq:Y1Y2} always satisfies 
the SLLN. So whether or not the SLLN holds for $Y_0(t)$, is completely determined by 
whether $Y_1(t)$ has a regular or irregular behavior, which has been investigated in the 
previous subsection.

\bprop\label{prop:small-box} 
Suppose that \eqref{eq:cond-lambda} holds. 
Then $Y_2$ from \eqref{eq:Y1Y2} satisfies
\beq\label{eq:lim-Y-Z}
\lim_{t \to \infty} \frac{Y_2(t)}{t}= m \quad \text{a.s.}
\eeq
\eprop

We first prove the SLLN for the special sequence $t_n=n$. If we have enough moments, this 
follows from standard moment bounds.
\begin{lemma} \label{lem:asYsubseq}
Let $t_n$ be a nondecreasing sequence tending to infinity. 
\benu
\item
If $m_\la(1+2/d) < \infty$, then 
\begin{equation} \label{eq:asYsubseq}
\lim_{n \to \infty} \frac{Y_0(t_n)}{t_n} = m \quad \text{a.s.} 
\end{equation}
holds if for some $\varepsilon > 0$,
\[
\begin{cases}
\displaystyle	\sum_{n=1}^\infty t_n^{\varepsilon -9/4} < \infty, & \text{ for } d =1, \\
\displaystyle	\sum_{n=1}^\infty t_n^{\varepsilon - (1 + 2/d)} < \infty, & \text{ for } d 
\geq 2.
\end{cases}
\]
\item	If we only have $m_\la(\kappa) < \infty$ for some $\kappa \in (1, (1+ 2/d)\wedge 
2)$, then 
(\ref{eq:asYsubseq}) holds if
\[
\sum_{n=1}^\infty t_n^{- (\kappa -1) (1 + d/2)} < \infty.
\]
\eenu
\end{lemma}
\begin{proof}
Simple computation gives
\[
\int_0^t \int_{\R^d} g(s, y)^p \,\dd y \,\dd s =
(4 \pi)^{\frac{d}{2}(1-p)} p^{-d/2} \frac{2}{2 - d(p-1)} t^{1 - \frac{d}{2}(p-1)},
\]
when $0<p < 1 + 2/d$.
Using predictable versions of the classical Burkholder--Davis--Gundy inequalities (see  
\cite[Theorem 1]{MR}) we have for any $\alpha \in [1,2]$,
\begin{equation} \label{eq:MR-ineq}
\begin{split}
& \E[| Y_0(t) - m t | ^p] \\ 
&\quad \leq 
\begin{cases}
c_{\alpha, p} \displaystyle
\left( \int_0^t \int_{\R^d} \int_{\R} |g(s,y)z|^\alpha\,
\lambda( \dd z) \, \dd y\, \dd s\right)^{p/\alpha}, &
p\leq\al, \\
c_{\alpha, p} \lp\displaystyle
\left( \int_0^t \int_{\R^d} \int_\R |g(s,y)z|^\alpha \lambda( \dd z) \, \dd y\, \dd s 
\right)^{p/\alpha}
+ \int_0^t \int_{\R^d} \int_\R |g(s,y)z|^p \lambda( \dd z) \, \dd y\, \dd s\rp, &
p>\al,
\end{cases}
\end{split}
\end{equation}
with some constant $c_{\alpha, p}\in(0,\infty)$. For $d = 1$, choosing $\alpha = 2$, 
$p = 3 - \delta$, and $0<\delta < 1$, we obtain 
\[
\p ( |Y_0(t) - m t| > \varepsilon t) \leq 
\frac{ \E [|Y_0(t) - m t|^p] }{\varepsilon^p t^p }\leq
C t^{-9/4 + 3 \delta / 4}.
\]
For $d \geq 2$, choosing $\alpha = p = 1 + 2/d - \delta$ gives
\[
\p ( |Y_0(t) - m t| > \varepsilon t )  \leq 
\frac{ \E [|Y_0(t) -  m t|^p] }{\varepsilon^p t^p }  \leq C t^{-(1 + 2/d) + \delta (1 + 
d/2)}.
\]
If we choose $\delta > 0$ small enough, the statement follows from the first 
Borel--Cantelli lemma.

If only $\int_{\R} |z|^\kappa \lambda(\dd z) < \infty$ for some 
$\kappa \in (1, (1+ 2/d) \wedge 2)$,  choose $\alpha = p = \kappa$ 
and the statement follows similarly.
\end{proof}

In particular, if $m_\la(\kappa)<\infty$ for some $\kappa>(4+d)/(2+d)$, the SLLN holds 
for the sequence $n$. Actually, a much weaker condition suffices, but then the proof 
becomes more involved.
\blem\label{lem:n} If $m_\la(1+\eps)<\infty$ for some $\eps>0$, then for any $a>0$, 
\eqref{eq:asYsubseq} holds on the sequence $t_n=an$.
\elem
\bpr
For notational simplicity, let us take $a=1$. We decompose the heat kernel as 
$g(t,x)=g_1(t,x)+g_2(t,x)$ such that $0\leq g_1,g_2\leq g$, $g_1(t,x) = 0$ for $t\geq1$ 
or 
$|x|\geq1$, $g_1(t,x)=g(t,x)$ for $t\leq 1/2$ and $|x|\leq 1/2$, $g_1$ is smooth on 
$(0,\infty)\times \R^d$, and $g_2$ is smooth on $[0,\infty)\times\R^d$ (including the 
origin $(t,x)=0$). Accordingly, we define
\begin{align*}
Y^{(1)}(t)&=\int_0^t\int_{\R^d}\int_\R g_1(t-s,y)z\, \mu(\dd s,\dd y,\dd z),\\ 
Y^{(2)}(t)&=m_0t+\int_0^t\int_{\R^d}\int_\R  g_2(t-s,y)z\,\mu(\dd s,\dd y,\dd z),
\end{align*}
such that $Y_0= Y^{(1)}+Y^{(2)}$.
By Example \ref{ex:1} (1), the series in \eqref{eq:t-cond-0} is finite. Thus,
applying Proposition~\ref{prop:rec-close} to the Lévy noise obtained by replacing all 
negative jumps of $\La$ by positive jumps of the same absolute value, we derive that 
$$
|Y^{(1)}(n)| \leq \sum_{\tau_i \in (n-1,n],\, |\eta_i| \leq 1 } g(n - \tau_i, 
\eta_i)|\zeta_i| = o(n) \quad \text{a.s.}
$$
Hence, by the second part of Lemma~\ref{lem:asYsubseq}, we have 
$$
\lim_{n\to\infty} 
\frac{Y^{(2)}(n^p)}{n^p} = \lim_{n\to\infty} \frac{Y_0(n^p)}{n^p} = m\quad\text{a.s.}
$$
if $p>2\eps^{-1}/(d+2)$.

For the subsequent argument, upon considering the drift, the positive and the negative 
jumps separately, we may assume without loss of generality that $\La$ only has positive 
jumps and that $m_0=0$. Then, given $n\in\N$, we choose $k\in\N$ such that 
$k^p\leq n \leq (k+1)^p$, and derive  
\beq\label{eq:Y2n} \frac{Y^{(2)}(n)}{n} \leq \frac{Y^{(2)}(n)}{k^p} = \frac{(k+1)^p}{k^p} 
\lp \frac{Y^{(2)}((k+1)^p)}{(k+1)^p}-\frac{Y^{(2)}((k+1)^p)-Y^{(2)}(n)}{(k+1)^p} \rp. \eeq
Since $g_2(t-s,y)= g_2(0,y)+\int_s^t \partial_t g_2(r-s,y)\,\dd r=\int_s^t \partial_t 
g_2(r-s,y)\,\dd r$, we have 
\beq\label{eq:Y2}
Y^{(2)}(t) = \int_0^t \int_{\R^d} \int_0^\infty 
\lp\int_s^t \partial_t g_2(r-s, y) \,  \dd r\rp z\,   \mu(\dd s, \dd y, \dd z).  
\eeq

By Lemma~\ref{lem:hk} and the defining properties of $g_2$, we have 
\begin{equation} \label{eq:g2-bound}
\begin{split}
|\partial_t g_2(t,x)|&\leq C,\quad (t,x)\in [0,\infty)\times\R^d, \\
|\partial_t g_2(t,x)| &= |\partial_t g(t,x)| \leq 
\begin{cases} 
C|x|^{-d-2} &\text{if } |x|^2> 2dt, \text{ and } t>1 \text{ or } |x|>1, \\
C\displaystyle t^{-d/2-1}  &\text{if } |x|^2\leq 2dt, \text{ and } t>1 \text{ or } |x|>1.
\end{cases}
\end{split}
\end{equation}
Using these bounds, straightforward calculation shows that
\[
\int_0^t \int_{\R^d} \int_0^\infty 
 \int_s^t |\partial_t g_2(r-s, y)|z \,  \dd r \,
\dd s \, \dd y\, \lambda(\dd z) < \infty, 
\]
which implies that almost surely, $|\partial_t g_2(r-s, y)| z$ is $(\dd r \otimes \dd \mu)$-integrable. Therefore, we can apply Fubini's theorem to \eqref{eq:Y2} and obtain
\[
Y^{(2)}(t)= \int_0^t \lp \int_0^r\int_{\R^d}\int_0^\infty \partial_t 
g_2(r-s,y)z\,\mu(\dd s,\dd y,\dd z)\rp\,\dd r.
\]
Thus, for any $\delta>0$ and $q=(1+\eps)\wedge 2$, 
\beq\label{eq:prob} \begin{aligned}
&\bbp\lb \sup_{n\in[k^p,(k+1)^p]} \lv \frac{Y^{(2)}((k+1)^p)-Y^{(2)}(n)}{(k+1)^p} \rv  > 
\delta\rb\\
&\qquad\leq \bbp\lb \frac{1}{(k+1)^p} \int_{k^p}^{(k+1)^p} \lv \int_0^r 
\int_{\R^d}\int_0^\infty \partial_t g_2(r-s,y)z\,\mu(\dd s,\dd y,\dd z)  \rv \,\dd r  > 
\delta\rb\\
&\qquad\leq \frac{1}{\delta^q(k+1)^{pq}} \bbe\lb \lv \int_{k^p}^{(k+1)^p} \lv \int_0^r 
\int_{\R^d}\int_0^\infty \partial_t g_2(r-s,y)z\,\mu(\dd s,\dd y,\dd z) \rv \,\dd r \rv 
^q 
\rb\\
&\qquad\leq \frac{C}{k^{p+q-1}} \int_{k^p}^{(k+1)^p} \bbe\lb  \lv \int_0^r 
\int_{\R^d}\int_0^\infty \partial_t g_2(r-s,y)z\,\mu(\dd s,\dd y,\dd z) \rv^q\rb \,\dd r,
\end{aligned}
\eeq
where we used Jensen's inequality on the $\dd r$-integral to pass to the last line.

Using \cite[Theorem~1]{MR} (with $\alpha  = q$) on the compensated version of the Poisson integral in the last line of \eqref{eq:prob}, the hypothesis $m_\lambda(q) < \infty$, and the fact 
that $\partial_t g_2$ is a bounded function, we can bound the last expectation by a 
constant times
\beq\label{eq:qmom}\begin{aligned}
	&\int_0^r \int_{\R^d}\int_0^\infty |\partial_t g_2(r-s,y)z|^q\,\la(\dd z)\,\dd y\,\dd s + \lv\int_0^r \int_{\R^d}\int_0^\infty \partial_t g_2(r-s,y)z\,\la(\dd z)\,\dd y\,\dd s\rv^q\\
	&\qquad\leq C\lp m_\la(q) + m_\la(1)^q\rp\lp \lp \int_0^r \int_{\R^d} |\partial_t g_2(s,y)| \,\dd y\,\dd s\rp^q\vee 1\rp\\
	&\qquad\leq C\lp\lp\int_0^r \int_{\R^d} |\partial_t g_2(s,y)| \,\dd y\,\dd s\rp^q\vee 1\rp.
\end{aligned}\eeq
Therefore, by \eqref{eq:g2-bound},
\begin{align*}
\int_0^r\int_{\R^d} |\partial_t g_2(s,y)| \,\dd y\,\dd s 
&\leq C \lp \int_0^1 \int_{|y|\leq \sqrt{2d}} 1\,\dd y\,\dd s+ \int_0^1 \int_{|y|>\sqrt{2d}} |y|^{-d-2} \,\dd y\,\dd 
s\right.\\
&\quad\left.+\int_1^r \int_{|y|> \sqrt{2ds}} |y|^{-d-2}\,\dd y\,\dd s+\int_1^r \int_{|y|\leq \sqrt{2ds}} s^{-d/2-1} \,\dd 
y\,\dd s \rp\\
&\leq C\lp 1+\int_1^r s^{-1}\,\dd s \rp\leq C\log^+ r.
\end{align*}
 
Inserting this back into \eqref{eq:qmom} and \eqref{eq:prob}, we conclude that 
\[ 
\bbp\lb \sup_{n\in[k^p,(k+1)^p]} \lv \frac{Y^{(2)}((k+1)^p)-Y^{(2)}(n)}{(k+1)^p} \rv  > 
\delta\rb \leq \frac{C}{k^{p+q-1}} \int_{k^p}^{(k+1)^p} (\log^+ r)^q \,\dd r \leq 
Ck^{-q} (\log^+ k)^q, 
\]
which is summable in $k$. The Borel--Cantelli lemma implies that almost surely,
\[
\sup_{n\in[k^p,(k+1)^p]} \lv \frac{Y^{(2)}((k+1)^p)-Y^{(2)}(n)}{(k+1)^p} \rv 
\longrightarrow 0 
\]
as $k\to\infty$. Thus, by \eqref{eq:Y2n}, 
\[ \limsup_{n\to\infty} \frac{Y^{(2)}(n)}{n} \leq m\quad\text{a.s.} \]
The reverse relation for the limit inferior is shown similarly.
\epr

\bpr[Proof of Proposition~\ref{prop:small-box}] 
By separating drift, positive jumps and negative jumps, it suffices to consider the case 
$m_0=0$ and $\la((-\infty,0))=0$.  Let $\theta > 0$
and choose $\delta > 0$ such that $\delta/(1-\delta) \leq \delta(\theta)$, where 
$\delta(\theta)$ is given in Lemma \ref{lem:heat-kernel-ineq}. Furthermore, define $s_n = 
\delta n$ and $k=\lfloor t/\delta\rfloor$ such that $s_k \leq t < s_{k+1}$. 
Because $(t-s_k)/(s_k-\tau_i) \leq (t - s_k)/(s_k - (t - 1)) \leq \delta/(1-\delta) \leq 
\delta(\theta)$ if $\tau_i\leq t-1$, and $|t-s_k|\leq \delta\leq \delta(\theta)$, we 
derive from
Lemma~\ref{lem:heat-kernel-ineq}, 
\begin{equation} \label{eq:I1-1-gen}
\begin{split}
Y_2(t)&= \sum_{\tau_i \leq t -1 } g(t - \tau_i, \eta_i)\zeta_i +
\sum_{\tau_i \in (t-1,t],\,|\eta_i|>1 } g(t - \tau_i, \eta_i)\zeta_i  \\
&\geq (1- \theta) 
\sum_{\substack{\tau_i \leq t - 1\,\text{or}\\ \tau_i \in (t-1,s_k],\,|\eta_i|>1}} 
g(s_k - \tau_i, \eta_i) \zeta_i \\
&= (1-\theta) \lp Y_0(s_k) -\sum_{\tau_i \in(t-1,s_k],\,|\eta_i|\leq 1} 
g(s_k-\tau_i,\eta_i)\zeta_i  \rp.
\end{split}
\end{equation}
The last sum is bounded by $Y_1(s_k)$, which is $o(s_k)$ by 
Proposition~\ref{prop:rec-close}. Similarly, we have
\begin{equation} \label{eq:I1-2-gen}
Y_0(s_{k+1}) \geq (1-\theta) Y_2(t) +Y_1(s_{k+1}).
\end{equation}
Combining \eqref{eq:I1-1-gen} and \eqref{eq:I1-2-gen} with 
Lemma~\ref{lem:n}, we obtain \eqref{eq:lim-Y-Z}.
\epr

\subsection{Proof of the main results}\label{sect:mr}

\begin{proof}[Proof of Theorem~\ref{thm:cont}.] For the first part, observe that we have $\int_1^\infty 1/(f(t)\vee t)\,\dd t =\infty$ by Lemma~\ref{lem:simple}. So Proposition~\ref{prop:Y1-cont} implies that 
	\[ \limsup_{t\to\infty} \frac{Y_1(t)}{f(t)\vee t} = \infty \qquad \lp \text{resp.,}\quad \limsup_{t\to\infty} \frac{Y_1(t)}{f(t)\vee t}=-\infty \rp. \] Moreover, by Proposition~\ref{prop:small-box}, 
	\[ \limsup_{t\to\infty} \lv \frac{Y_2(t)}{f(t)\vee t}\rv <\infty, \]
which implies the claim of Theorem~\ref{thm:cont} (1) for the function $f(t)\vee t$, and hence also for $f(t)$.

Next, the second part of theorem is an easy consequence of fact that $t/f(t)\to 0$ by Lemma~\ref{lem:simple} together with Propositions~\ref{prop:Y1-cont} and \ref{prop:small-box}. For the third 
part, let us assume that $\la((-\infty,0))=0$ (the proof in the case $\la((0,\infty))=0$ 
is analogous). 
Then Proposition~\ref{prop:small-box} combined with 
\eqref{eq:liminfsup} implies the statement.
\end{proof}

\begin{proof}[Proof of Theorem~\ref{thm:WLLN}]
We use the moment bound \eqref{eq:MR-ineq}.
For $p \leq 1$ choose $\alpha \in ( 1, 1 + ( \varepsilon \wedge 2/d ) )$, for 
$p \in (1, (1 + \varepsilon) \wedge 2] \cap (1, 1 + 2/d)$ choose $\alpha = p$, while for 
$p \in (2,3) \cap (2, 1 + \varepsilon]$ choose $ \alpha = 2$. Then, we obtain
\[
\frac{ \E [|Y_0(t) -  m t|^p] }{t^p }  \leq 
\begin{cases} 
C t^{- (1-\alpha^{-1}) p ( 1 + \frac{d}{2} ) } & \text{if } p \leq 1,\\
Ct^{-(p-1)(1+\frac d2)} & \text{if } p\in (1,(1+\eps) \wedge 2] \cap (1,1+\frac 2d),\\ C(t^{-\frac 34 p}+t^{-\frac 32(p-1)}) & \text{if } d=1 \text{ and } p\in (2,1+\eps]\cap(2,3),\end{cases}  \]
	and the theorem follows.
\end{proof}

\begin{proof}[Proof of Theorem~\ref{thm:Yf-ds}.] 
For the first part of the theorem, let us assume \eqref{eq:t-cond-0} and write
\[
\frac{Y_0(t_n)}{f(t_n)} = \frac{Y_1(t_n)}{f(t_n)} + 
\frac{Y_2(t_n)}{t_n} \frac{t_n}{f(t_n)}. 
\]
By Proposition~\ref{prop:small-box} and the assumption 
$\liminf_{n \to \infty} f(t_n)/t_n > 0$, the second term is bounded, while the first one is not by
Proposition \ref{prop:ssconverse}. For the second part, if
the series in \eqref{eq:t-cond-0} 
converges, Proposition~\ref{prop:rec-close} and \ref{prop:small-box} imply that 
\[ \limsup_{n\to\infty} \frac{Y_0(t_n)}{f(t_n)} = \lim_{n\to\infty} 
\frac{Y_2(t_n)}{f(t_n)} = m\kappa\quad \text{a.s.} \]
\end{proof}

\begin{remark}\label{rem:growth}
	Because of the linear growth of the expectation, it is natural to consider only functions 
	$f$ for which $\liminf_{n \to \infty} f(t_n) / t_n > 0$ in the first part of Theorem~\ref{thm:Yf-ds}. Assuming that $m_\lambda(1+2/d) < 
	\infty$, it is easy to see that for the sequence $t_n = n^p$ with $p \in (d/(d+2), 
	d/(d+1))$, we have by Propositions \ref{prop:rec-close} and \ref{prop:ssconverse} and
	Corollary \ref{cor:2}, 
	\[
	\limsup_{n \to \infty} \frac{Y_1(t_n)}{\sqrt{t_n}} = \infty, \quad 
	\limsup_{n \to \infty} \frac{Y_1(t_n)}{t_n} = 0.
	\]
	Therefore, with $f(t) = \sqrt{t}$, we obtain
	\[
	\frac{Y_0(t_n)}{\sqrt{t_n}} =
	\sqrt{t_n} \left( \frac{Y_1(t_n)}{t_n} + \frac{Y_2(t_n)}{t_n} \right),
	\]
	and the limit is either $+\infty$ or $-\infty$, depending on the sign of the mean $m$. So in this example, the large time behavior is dominated by the mean and not (only) by the jumps of the noise.
\end{remark}

\begin{proof}[Proof of Theorem~\ref{thm:cs-g}.]
Let $Y^\pm_0(t) = \sum_{\pm \zeta_i>0} g(t-\tau_i,\eta_i)\zeta_i$ (\emph{without} the 
multiplicative nonlinearity $\si$). Furthermore, recall the meaning of the constants in 
\eqref{eq:si}. Then, in order to prove (1), we simply bound ($C=k_1$ if $m_0\geq0$, and 
$C=k_2$ if $m_0<0$)
\beq\label{eq:greater} \frac{Y_0(t)}{f(t)} \geq k_1 \frac{Y^+_0(t)}{f(t)} 
+ k_2 \frac{Y^-_0(t)}{f(t)} +Cm_0\frac{t}{f(t)}. \eeq
Thus, if $\la((0,\infty))>0$ and $\int_1^\infty 1/f(t) \,\dd t= \infty$, we have 
\[ \limsup_{t\to\infty} \frac{Y_0(t)}{f(t)} =\infty\]
by the corresponding statement in the additive case (Theorem~\ref{thm:cont}). The claim on the limit inferior if $\la((-\infty,0))>0$ holds by symmetry.

Conversely, if $\int_1^\infty 1/f(t)\,\dd t<\infty$, then 
\beq\label{eq:help2} \lv \frac{Y_0(t)}{f(t)}\rv \leq |m_0|k_2\frac{t}{f(t)}+k_2\frac{Y^+_0(t)+|Y^-_0(t)|}{f(t)} \longrightarrow 0\quad\text{a.s.} \eeq
by Theorem~\ref{thm:cont}.

For (2), it suffices to bound $Y_0(t)/t \leq m_0C' +k_1Y^-_0(t)/t \leq m_0 C'$ (resp., $Y_0(t)/t \geq m_0C +k_1Y^+_0(t)/t \geq m_0 C$), where $C=k_1$ and $C'=k_2$ if $m_0\geq0$, and $C=k_2$ and $C'=k_1$ if $m_0<0$.
\end{proof}

\begin{proof}[Proof of Theorem~\ref{thm:WLLN-g}]
	Use the first estimate in \eqref{eq:help2} with $f(t)=t$ and then Theorem~\ref{thm:WLLN}.
\end{proof}

\begin{proof}[Proof of Theorem~\ref{thm:ds-g}.] If \eqref{eq:t-cond-0} holds, \eqref{eq:aux2} follows from the estimate \eqref{eq:greater}, the fact that we have $\limsup_{n\to\infty} Y^+_0(t_n)/f(t_n) = \infty$ by Theorem~\ref{thm:Yf-ds}, and the independence of $Y^+_0$ and $Y^-_0$ (cf.\ the argument given after \eqref{eq:help}). The same reasoning applies if \eqref{eq:t-cond-1} holds. For the second statement of Theorem~\ref{thm:ds-g}, let us only consider the case where the series in \eqref{eq:t-cond-0} is finite. Then the assertion is deduced from the bound $Y_0(t_n)/f(t_n)\leq m_0C't_n/f(t_n) + k_2Y^+_0(t_n)/f(t_n)$ and Theorem~\ref{thm:Yf-ds} (2) applied to $Y^+_0$.
\end{proof}

\begin{proof}[Proof of Theorem~\ref{thm:B}] 
The theorem follows from the general asymptotic theory of Gaussian processes 
in \cite{Watanabe70}. Let $v^2(t)=\bbe[Y_0(t)^2]$ and $\rho(t,s) = 
\bbe[Y_0(t)Y_0(s)]/(v(t)v(s))$ be the variance and the correlation function of $Y_0$, 
respectively. Then by \cite[Lemma~2.1]{Swanson07}, 
\[ v^2(t)=\sqrt{\frac t {2\pi}},\quad 
\rho(t,t+h)=\frac{\sqrt{2+h/t}-\sqrt{h/t}}{(4(1+h/t))^{1/4}},\quad t>0,~h\geq0. \]
Notice that \cite{Swanson07} considers the heat equation with a factor $1/2$ in front of 
the Laplacian, but the moment formulae can be easily transformed to our situation by a 
scaling argument. From these identities, it is easy to verify that 
\[ \frac{1-\rho(t,t+h)}{(h/t)^{1/2}} \leq \frac1{\sqrt{2}}, \quad \rho(t,ts)\log s = 
\frac{\sqrt{s+1}-\sqrt{s-1}}{(4s)^{1/4}}\log s,\quad t>0,~h\geq0,~s>1,\]
which means that condition (C.1) with $\al=1/2$ as well as condition (C.2) in 
\cite{Watanabe70} are satisfied. Also (C.1') in this reference holds true because 
$\rho(t,t+h)\leq 1-\frac12 (h/t)^{1/2}$ for small values of $h/t$, and $\rho(t,t+h)$ 
decreases in $h/t$. 

Therefore, \cite[Theorem~6]{Watanabe70} is applicable, and by considering the functions 
$\psi(t)=\sqrt{K\log\log t}$, with $K>0$, we deduce that almost surely
\[ \limsup_{t\to\infty} \frac{Y_0(t)}{\sqrt{\sqrt{\frac{t}{2\pi}}K\log\log t}} \ 
\begin{cases} \leq 1&\text{for } K>2,\\ \geq 1&\text{for } K\leq2,\end{cases} \]
from which the statement of Theorem~\ref{thm:B} follows.
\end{proof}

\medskip
\noindent \textbf{Acknowledgments.}
We are grateful to an anonymous referee for useful suggestions that have led to several improvements of the article. Furthermore, we would like to thank Davar Khoshnevisan for discussing the case of multiplicative Gaussian noise with us and for pointing out the reference \cite{Bertini99} to us.
CC thanks the Bolyai Institute at the University of Szeged for its hospitality during his visit.
PK's research was supported by the János Bolyai Research Scholarship of the 
Hungarian Academy of Sciences, by the NKFIH grant FK124141, and 
by the EU-funded Hungarian grant EFOP-3.6.1-16-2016-00008.

\bibliographystyle{plain}
\bibliography{almost-sure} 
\end{document}